\documentclass[english,11pt]{smfart}

\usepackage{appendix}
\usepackage{etex}

\usepackage{amsbsy}
\usepackage{amsmath,amsfonts,amssymb,amsthm,mathrsfs,mathtools}

\usepackage{bm}

\usepackage[a4paper,vmargin={3cm,3cm},hmargin={3.5cm,3.5cm}]{geometry}
\usepackage[font=sf, labelfont={sf,bf}, margin=1cm]{caption}
\usepackage{graphicx}
\usepackage{epsfig}%pour les eps
\usepackage{latexsym}%encore des symboles
\usepackage{xcolor}
\usepackage{ae,aecompl}
\usepackage{soul,framed}
\usepackage{comment}

\usepackage{xcolor}
\usepackage[pdfpagemode=UseNone,bookmarksopen=false,colorlinks=true,urlcolor=blue,citecolor=blue,citebordercolor=blue,linkcolor=blue]{hyperref}
\usepackage{smfhyperref}
\usepackage[capitalize]{cleveref}

\usepackage{pstricks}
\usepackage{enumerate}
\usepackage{tikz,animate,media9}						%%%%%%%%%%
\usepackage{todonotes}
\usepackage{pifont}
\usepackage{bm,marvosym}
\usepackage{algorithm}
\usepackage{algorithmic}

\usepackage{caption}
\usepackage{subcaption}

\usepackage{bbm}

\usepackage{tcolorbox}

\definecolor{aleacolor}{rgb}{0.16,0.59,0.78}

\hypersetup{
	breaklinks,
	colorlinks=true,
	linkcolor=aleacolor,
	urlcolor=aleacolor,
	citecolor=aleacolor}

\newcount\colveccount
\newcommand*\colvec[1]{
	\global\colveccount#1
	\begin{pmatrix}
		\colvecnext
	}
	\def\colvecnext#1{
		#1
		\global\advance\colveccount-1
		\ifnum\colveccount>0
		\\
		\expandafter\colvecnext
		\else
	\end{pmatrix}
	\fi
}

\newcommand{\ndN}{\mathbb{N}}

\newcommand{\ndR}{\mathbb{R}}
\newcommand{\ndC}{\mathbb{C}}

\renewcommand{\Pr}[1]{\mathbb{P}(#1)}

\newcommand{\Prb}[1]{\mathbb{P}\left(#1\right)}

\newcommand{\Ex}[1]{\mathbb{E}[#1]}

\newcommand{\convdis}{\,{\buildrel d \over \longrightarrow}\,}
\newcommand{\convd}{\,{\buildrel d \over \longrightarrow}\,}

\newcommand{\convp}{\,{\buildrel p \over \longrightarrow}\,}

\newcommand{\convas}{\,{\buildrel a.s. \over \longrightarrow}\,}

\newcommand{\He}{\mathrm{H}}
\newcommand{\he}{\mathrm{h}}

\newcommand{\Di}{\mathrm{D}}

\newcommand{\norm}[1]{\left\lVert#1\right\rVert}

\newcommand{\cA}{\mathcal{A}}
\newcommand{\cB}{\mathcal{B}}
\newcommand{\cC}{\mathcal{C}}
\newcommand{\cD}{\mathcal{D}}
\newcommand{\cE}{\mathcal{E}}
\newcommand{\cF}{\mathcal{F}}

\newcommand{\cH}{\mathcal{H}}
\newcommand{\cI}{\mathcal{I}}

\newcommand{\cL}{\mathcal{L}}
\newcommand{\cM}{\mathcal{M}}
\newcommand{\cN}{\mathcal{N}}

\newcommand{\cP}{\mathcal{P}}
\newcommand{\cQ}{\mathcal{Q}}

\newcommand{\cS}{\mathcal{S}}

\newcommand{\mC}{\mathsf{C}}
\newcommand{\mD}{\mathsf{D}}

\newcommand{\mM}{\mathsf{M}}

\newcommand{\mQ}{\mathsf{Q}}

\newcommand{\mS}{\mathsf{S}}
\newcommand{\mT}{\mathsf{T}}

\newcommand{\mX}{\mathsf{X}}

\newtheorem{theorem}{Theorem}[section]

\newtheorem{corollary}[theorem]{Corollary}
\newtheorem{proposition}[theorem]{Proposition}
\newtheorem{lemma}[theorem]{Lemma}

\newtheorem{definition}[theorem]{Definition}
\newtheorem{remark}[theorem]{Remark}

\numberwithin{equation}{section}

\title{\textbf{The scaling limit of random cubic planar graphs}}
\date{}

\author{Benedikt Stufler}
\address[Benedikt Stufler]{Vienna University of Technology}
\email{benedikt.stufler at tuwien.ac.at}

\begin{document}

\vspace {-0.5cm}

\begin{abstract}
	We study the random simple connected cubic planar graph $\mathsf{C}_n$ with an even number $n$ of vertices. We show that the Brownian map arises as Gromov--Hausdorff--Prokhorov scaling limit of $\mathsf{C}_n$  as $n \in 2 \ndN$ tends to infinity, after rescaling distances by $\gamma n^{-1/4} $ for a specific constant $\gamma>0$.
\end{abstract}

%\tableofcontents

\maketitle

\section{Introduction}

Planar graphs are graphs that admit a crossing-free embedding into the $2$-sphere. We call such a graph cubic, if each vertex is adjacent to precisely three edges. Cubic planar graphs and related classes have been enumerated by \cite{zbMATH05122852, zbMATH07213288} via analytic and combinatorial methods. The average number of perfect matchings in a random cubic planar graph was determined in~\cite{zbMATH06639396}. The  work~\cite{stufler2022uniform} established a Uniform Infinite Cubic Planar Graph as their quenched local limit. The paper \cite{10.5565/PUBLMAT6612213} determined the typical number of triangles in $3$-connected cubic planar graphs. Related research directions concern $4$-regular planar graphs~\cite{zbMATH06827273}, cubic graphs on general orientable surfaces~\cite{zbMATH06841874}, and cubic planar maps~\cite{coreprep}. In particular,~\cite{zbMATH06556653} determined the geodesic two- and three-point functions of random cubic planar
maps, after assigning independent random lengths with an exponential distribution to each edge.

For  any even number $n \ge 4$ we let $\mC_n$ denote the uniform random simple connected cubic planar graph  $\mC_n$  with $n$ labelled vertices and hence $3n/2$ edges. The graph distance on $\mC_n$ is denoted by  $d_{\mC_n}$. We let $\mu_{\mC_n}$ denote the uniform measure on the vertex set of $\mC_n$. We let $(\mathbf{M}, d_{\mathbf{M}}, \mu_{\mathbf{M}})$ denote the Brownian map established independently in~\cite{MR3112934} and \cite{MR3070569}. See  Figure~\ref{fi:bm} for an illustration.  Its construction is briefly recalled in Section~\ref{eq:ghbm}. Our main result shows that the Brownian map describes the asymptotic global geometric shape of $\mC_n$.

\begin{theorem}
	\label{te:main}
	There exists a constant $\gamma>0$ such that
	\begin{align}
		\label{eq:mapconverge}
		\left(\mC_n, \gamma n^{-1/4} d_{\mC_n}, \mu_{\mC_n} \right) \convdis (\mathbf{M}, d_{\mathbf{M}}, \mu_{\mathbf{M}})
	\end{align}
	in the Gromov--Hausdorff--Prokhorov sense as $n \in 2\ndN$ tends to infinity.
\end{theorem}

\begin{figure}[t]
	\centering
	\begin{minipage}{\textwidth}
		\centering
		\includegraphics[width=1.0\linewidth]{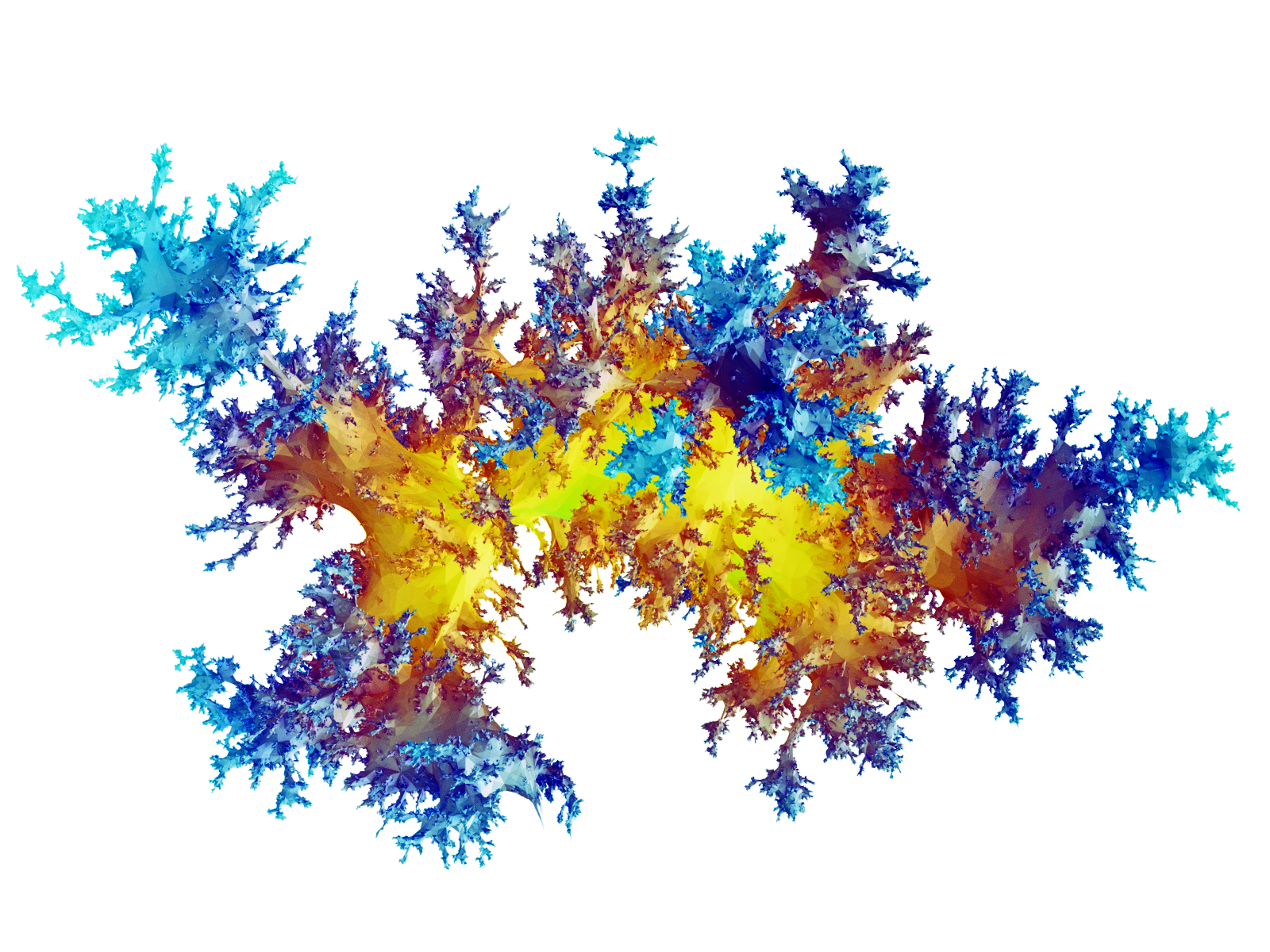}
		\caption[asdf]{The Brownian map, approximated by a random simple triangulation  of the two-dimensional sphere with a million triangles.\protect\footnotemark~In the proof of Theorem~\ref{te:main} we show that random cubic planar graphs may be approximated by random simple triangulations in the Gromov--Hausdorff--Prokhorov metric.}
		\label{fi:bm}
	\end{minipage}
\end{figure}

\footnotetext{The colours in the drawing represent closeness centrality in the dual map. The coordinates of the vertices have been selected according to a spring-electrical embedding algorithm. The triangulation was generated using the author's open source software \texttt{simtria}, freely available on the github: \url{https://github.com/BenediktStufler/simtria}. }

The Brownian map is known to be universal in the sense that it arises as scaling limit of random elements of seemingly independent classes of planar maps. Specifically, quadrangulations~\cite{MR3070569}, triangulations and $2q$-angulations for $q \ge 2$~\cite{MR3112934},  bipartite planar maps~\cite{abraham2016}, unrestricted planar maps~\cite{MR3256874}, simple triangulations and quadrangulations~\cite{MR3706731}, $(2q+1)$-angulations for $q \ge 2$~\cite{AHL_2021__4__653_0},  quadrangulations and their cores~\cite{MR3729639},  specific critical Boltzmann planar maps~\cite{Mar18b}, and biconditioned planar maps~\cite{KM21}. 

It appears that less is known about the geometry of graph classes as opposed to maps. The diameter of random graphs from subcritical classes was studied in~\cite{MR3184197}. These models have a more tree-like shape as shown in the work~\cite{zbMATH06653517}, which establishes  Aldous' Brownian continuum random tree~\cite{MR1085326} as their scaling limit. Graph classes that are critical in a specific sense~\cite{MR3068033} are expected to exhibit different asymptotic shapes. The strongest indication being large deviation bounds for the diameter of unrestricted planar graphs~\cite{MR2735332}. Theorem~\ref{te:main} provides a first scaling limit of a critical graph class. An independent proof of this result is given in forthcoming work~\cite{emt2022} via a different approach.

An important bridge between the study of graphs and maps is the scaling limit of simple triangulations~\cite{MR3706731}, since these maps are $3$-connected and Whitney's theorem ensures that each $3$-connected planar graph has precisely two planar  embeddings. That is, simple planar triangular maps and graphs are equivalent models.
The proof of Theorem~\ref{te:main} builds upon this result. This is possible since the uniform cubic planar graph $\mC_n$ has a giant $3$-connected component~\cite[Thm. 1.2]{stufler2022uniform} with roughly $\kappa n$ vertices for $\kappa \approx 0.85085$. This way, $\mC_n$  may be viewed as the result of blowing up the edges of the $3$-connected core by networks, so that the total number of vertices in the core and in the network components equals~$n$.  These network components are hence not independent, as knowledge of a linear number of network components places a bias on the size of the core. However,  we are going to argue that Theorem~\ref{te:main} actually follows from a scaling limit of a modified model where the network components of the $3$-connected core are replaced by independent copies of a specific Boltzmann network. We show that the maximal diameter of the attached components is negligible compared to the diameter of the entire graph. Hence the modified model has the geometry of a first-passage percolation metric on the $3$-connected core, with random positive integer-valued light-tailed link weights. It is known that $3$-connected cubic planar graphs are precisely the dual graphs of simple triangulations. Thus, using the scaling limit of simple triangulations~\cite{MR3706731} and suitable adaptions~\cite{simplefpp} of concentration inequalities for first-passage percolation on the dual of random  triangulations~\cite{zbMATH07144469} yield a scaling limit for the modified model and hence complete the proof of Theorem~\ref{te:main}.

\subsection*{Notation}
%We set $[n] = \{1, \ldots, n\}$ for all non-negative integers $n$. %The total variational distance of two random variables $X$ and $Y$ is denoted by $d_{\textsc{TV}}(X,Y)$.

All random variables occurring in this work are assumed to be defined on a common probability space whose measure we denote by $\mathbb{P}$. All unspecified limits are as $n \to \infty$, possibly along a subset of the positive integers.

% For two sequences $(X_n)_{n \ge 1}$ and $(Y_n)_{n \ge 1}$ of random variables with values in a common Polish space  we write $X_n \atv Y_n$ if their total variation distance $d_{\textsc{TV}}(X_n,Y_n)$ tends to zero.

We say an event that depends on $n$ holds \emph{with high probability} if its probability tends to $1$ as $n$ becomes large. 	An event holds with \emph{very high probability}, if the probability for its complement may be bounded by $C \exp(-c n^\delta)$ for some constants $C,c,\delta>0$ that do not depend on $n$.

Convergence in probability and distribution are denoted by $\convp$ and $\convd$, almost sure convergence is denoted by $\convas$. %Weak convergence of measures is denoted by~$\Longrightarrow$.

For any sequence $a_n>0$ we let $o_p(a_n)$ denote an unspecified random variable $Z_n$ such that $Z_n / a_n \convp 0$. Likewise $O_p(a_n)$ is a random variable $Z_n$ such that $Z_n / a_n$ is stochastically bounded. 

Given a metric space $(M, d_M)$, the distance of a point $x \in M$ from a subset $A \subset M$ is the infimum of distances between $x$ and points from $A$. We use the notation
\[
	d_M(x,A) = \inf_{y \in A} d_M(x,y).
\]
The degree of a vertex $v$ in a graph $G$ is denoted by $d_G(v)$. If $G$ is connected, $d_G(u,v)$ denotes the graph distance between vertices $u,v$ of $G$.

\section{Preliminaries}
\label{eq:ghbm}

\subsection{Gromov--Hausdorff--Prokhorov distance and related metrics}

%For any metric space $(X,d_X)$ and any  $\alpha > 0$ we let $\alpha X$ denote the rescaled space $(X, \alpha d_X)$.  

The \emph{Hausdorff-distance} is a metric on the collection of compact subsets of a metric space $(X,d_X)$. It is given by
\[
d_{\textsc{H}}(A,B) = \inf\{\epsilon > 0 \mid A \subset U^X_{\epsilon}(B), B \subset U^X_{\epsilon}(A)\} 
\]
with \[
U^X_\epsilon(A) = \{x \in X \mid d_X(x,A)< \epsilon\}\]
 denoting the \emph{$\epsilon$-hull} of a subset $A$. That is, the set of all points in $X$ with distance less than $\epsilon$ from $A$.

The \emph{Prokhorov-distance} between  probability measures $\mu$ and $\nu$ on the Borel $\sigma$-algebra $\mathfrak{B}(X)$ on $X$ is given by
\[
d_{\textsc{P}}(\mu, \nu) = \inf\{ \epsilon >0 \mid \forall A \in \mathfrak{B}(X): \mu(A) \le \nu(U^X_\epsilon(A)) + \epsilon \}.
\]
The Prokhorov-distance satisfies the axioms of a metric, and in particular \[
d_{\textsc{P}}(\mu, \nu) = d_{\textsc{P}}(\nu, \mu).
\] See~\cite[Thm. 11.3.1]{dudley_2002} for a detailed justification.

Let  $(X, d_X, \mu_X)$, $(Y, d_Y, \mu_Y)$ be compact metric spaces endowed with Borel probability measures. 

The \emph{Gromov--Hausdorff} (GH) distance between  $(X, d_X)$ and $(Y, d_Y)$ is given by the minimal Hausdorff distance of isometric embeddings of $X$ and $Y$ into any common space. That is,
\[
d_{\textsc{GH}}(X,Y) = \inf_{\phi_X, \phi_Y} d_{\textsc{H}}(\phi_X(X),\phi_Y(Y)),
\]
with the indices $\phi_X, \phi_Y$ ranging over all isometric embeddings $\phi_X: X \to Z$ and $\phi_Y: Y \to Z$ for all choices of metric spaces $(Z, d_Z)$. Letting $\mu_X \phi_X^{-1}$ and $\mu_Y \phi_Y^{-1}$ denote the push-forward measures, the \emph{Gromov--Hausdorff--Prokhorov} (GHP) distance is defined similarly by
\begin{multline*}
d_{\textsc{GHP}}((X, d_X,  \mu_X),(Y, d_Y, \mu_Y)) \\= \inf_{\phi_X, \phi_Y} \max(d_{\textsc{H}}(\phi_X(X),\phi_Y(Y)), d_{\textsc{P}}(\mu_X \phi_X^{-1},\mu_Y \phi_Y^{-1})).
\end{multline*}
We say $(X, d_X, \mu_X)$ and $(Y, d_Y, \mu_Y)$ are \emph{isometry-equivalent}, if there exists a bijective isometry $\phi: X \to Y$ such that $\mu_Y = \mu_X \phi^{-1}$. The GHP distance induces a metric on the set $\mathfrak{K}$ of (representatives of) equivalence classes of compact metric spaces endowed with Borel probability measures.\footnote{The equivalence classes are obviously not sets, and hence an infinite collection of such classes is not a set. For this technical reason, we need to work with a collection of representatives instead.} See for instance~\cite[p. 762]{zbMATH05306371} and~\cite{MR2571957}. The metric space $(\mathfrak{K}, d_{\mathrm{GHP}})$ is complete and separable, see~\cite[Sec. 2]{zbMATH05043271}, ~\cite[Sec. 6]{MR2571957} and~\cite[Ch. 7]{MR1835418}. See also~\cite{MR3035742} for the related local GHP-distance, and~\cite{janson2020gromovprohorov} for properties of the related Gromov--Prokhorov distance.

%We refer the reader to~\cite[Ch. 27]{zbMATH05306371}, \cite[Ch. 7]{MR1835418},  \cite[Sec. 6]{MR2571957},~\cite{MR3035742} and the recent survey~\cite{janson2020gromovprohorov} for further detailed expositions.

Our main result, Theorem~\ref{te:main}, states a distributional limit of random elements of the space $(\mathfrak{K}, d_{\mathrm{GHP}})$. In its proof, we will use a specific characterization of the GHP-distance:

A relation $R \subset X \times Y$ between $X$ and $Y$ is a \emph{correspondence}, if every $x \in X$ corresponds to at least one $y \in Y$ and vice versa. We let $R(X,Y)$ denote the collection of correspondences between $X$ and $Y$.
The \emph{distortion} of a correspondence is defined by
\[
\mathrm{dis}(R) = \sup\{ | d_X(x,x') - d_Y(y,y') | \mid (x,y), (x',y') \in R \}.
\]
 A \emph{coupling} between $\mu_X$ and $\mu_Y$ is a Borel measure $\nu$ on $X \times Y$ with marginal distributions $\mu_X$ and $\mu_Y$. We let $M(\mu_X, \mu_Y)$ denote the collection of couplings between $\mu_X$ and $\mu_Y$. This allows for the following description of the GHP-distance:
\begin{proposition}[{\cite[Prop. 6]{MR2571957}}]
	\label{pro:ghpcoupling}
	Let  $(X, d_X, \mu_X)$, $(Y, d_Y, \mu_Y)$ be compact metric spaces endowed with Borel probability measures. The Gromov--Hausdorff--Prokhorov distance between these spaces is equal to the infimum of all $\epsilon>0$ for which there is a correspondence $R \in R(X,Y)$ satisfying
	\[
		\frac{1}{2} \mathrm{dis}(R) \le \epsilon
	\]
	and coupling $\nu \in M(X,Y)$ with
	\[
	\nu(R) \ge 1- \epsilon.
	\]
\end{proposition}

\subsection{The mass and radius of balls in GHP-convergent sequences}

We recall general properties of GHP-convergent sequences following closely the presentation in~\cite{stufler2022mass}.

\subsubsection{Balls with small diameter have small mass}

For any $\epsilon>0$ and any compact measured metric space $(X, d_X, \mu_X) \in \mathfrak{K}$ we define
\begin{align}
	\label{eq:deflambda}
	\lambda_\epsilon^X := \sup_{x \in X} \mu_X(B_\epsilon^X(x)).
\end{align}
Here 
\begin{align}
	\label{eq:openball}
	B_\epsilon^X(x) = \{y \in X \mid d_X(x,y) < \epsilon\}
\end{align} denotes the open $\epsilon$-ball in $X$. 

If $(X,d_X, \mu_X)$ is the Gromov--Hausdorff--Prokhorov limit of a sequence $(X_n, d_{X_n}, \mu_{X_n})_{n \ge 1}$ in $\mathfrak{K}$, then the following statements are equivalent:
\begin{enumerate}[\qquad 1)]
	\item There is no point $x \in X$ with $\mu_X(\{x\}) >0$.
	\item $\lim_{\epsilon \downarrow 0} \lambda_\epsilon^X=0$.
	\item $\lim_{\epsilon \downarrow 0} \limsup_{n \to \infty} \lambda_\epsilon^{X_n}=0$
\end{enumerate}
See~\cite[Cor. 2.3]{stufler2022mass} for a detailed justification. A version for random elements reads as follows:
\begin{lemma}[{\cite[Cor. 2.4]{stufler2022mass}}]
	\label{le:conconlambda}
	Let $(\mX,d_{\mX}, \mu_{\mX})$ and $(\mX_n, d_{\mX_n}, \mu_{\mX_n})_{n \ge 1}$ be random elements of $\mathfrak{K}$ satisfying
	\[
	(\mX_n, d_{\mX_n}, \mu_{\mX_n}) \convdis (\mX,d_{\mX}, \mu_{\mX}).
	\]
	Then the following statements are equivalent:
	\begin{enumerate}
		\item Almost surely there is no $x \in \mX$ with $\mu_\mX(\{x\})>0$.
		\item For all $\epsilon, \epsilon'>0$ there exist $\delta>0$ and $N>0$ such that for all $n \ge N$: $\Pr{\lambda_\delta^{\mX_n} > \epsilon} < \epsilon'$.
	\end{enumerate}
\end{lemma}

Hence, if the limit is diffuse, for large $n$ each open ball in $\mX_n$ with small diameter is likely to have small mass.

\subsubsection{Balls with small mass have small radius}
For any $\epsilon>0$ and any compact measured metric space $(X, d_X, \mu_X) \in \mathfrak{K}$ we define
\begin{align}
	\label{eq:def}
	\rho_\epsilon^X := \inf_{x \in X} \mu_X(C_\epsilon^X(x)).
\end{align}
Here 
\begin{align}
	\label{eq:mbal}
	C_\epsilon^X(x) = \{y \in X \mid d_X(x,y) \le \epsilon\}
\end{align} denotes the closed $\epsilon$-ball in $X$. 

A Borel probability measure has \emph{full support}, if any open non-empty set has positive measure.  If $(X,d_X, \mu_X)$ is the Gromov--Hausdorff--Prokhorov limit of a sequence $(X_n, d_{X_n}, \mu_{X_n})_{n \ge 1}$ in $\mathfrak{K}$, then the following statements are equivalent:
\begin{enumerate}[\qquad 1)]
	\item $\mu_X$ has full support.
	\item $\rho_\epsilon^X>0$ for all $\epsilon>0$.
	\item $\liminf_{n \to \infty} \rho_\epsilon^{X_n} >0$ for all $\epsilon>0$. 
\end{enumerate}
A detailed justification is provided in the survey \cite[Cor. 3.3]{stufler2022mass}. See also ~\cite[Lem. 15]{MR2571957}, which deduces statement 1) from the a priori weaker assumption that \[
\limsup_{n \to \infty} \inf_{x \in X_n} \mu_{X_n}(C_\epsilon^{X_n}(x))>0
\] for some $\epsilon_n$-dense subsets $X_n \subset X$ with $\epsilon_n \to 0$. A version for random elements reads as follows:
\begin{lemma}[{\cite[Cor. 3.4]{stufler2022mass}}]
	\label{co:fullrandom}
	Let $(\mX,d_{\mX}, \mu_{\mX})$ and $(\mX_n, d_{\mX_n}, \mu_{\mX_n})_{n \ge 1}$ be random elements of $\mathfrak{K}$ satisfying
	\[
	(\mX_n, d_{\mX_n}, \mu_{\mX_n}) \convdis (\mX,d_{\mX}, \mu_{\mX})
	\]
	as $n \to \infty$. Then the following statements are equivalent:
	\begin{enumerate}
		\item $\mu_\mX$ has almost surely full support.
		\item For all $\epsilon, \epsilon'>0$ there are $\delta, N>0$ such that for all $n \ge N$ 
		\[\Pr{\rho_\epsilon^{\mX_n} < \delta} < \epsilon'.\]
	\end{enumerate}
\end{lemma}

In other words, if the limit is almost surely fully supported, then for large $n$ each closed ball in $\mathsf{X}_n$ with small mass is likely to  have small radius.

\subsection{The Brownian map}
\label{sec:brownianmap}
We recall the construction of the Brownian map following~\cite{MR3112934,MR3070569}. The head of the Brownian snake $(\mathbf{e}, \mathbf{Z})$ is a random element of $\cC([0,1], \ndR^2)$. Here $\mathbf{e}$ denotes a Brownian excursion normalized to have duration one. Conditionally on $\mathbf{e}$, the random function $\mathbf{Z}$ is a continuous, centred Gaussian process with covariance \[
\mathbf{Cov}(\mathbf{Z}(s), \mathbf{Z}(t)) = \min \{ \mathbf{e}(x) \mid \min(s,t) \le x \le \max(s,t)\}, \qquad s,t \in [0,1].
\]
Let $d_{\mathbf{e}}$ denote the pseudo-distance on $[0,1]$ defined by
\[
d_{\mathbf{e}}(s,t) = \mathbf{e}(s) + \mathbf{e}(t) - 2 \min \{ \mathbf{e}(x) \mid \min(s,t) \le x \le \max(s,t)\}, \qquad s,t \in [0,1].
\]
Similarly, let $d_{\mathbf{Z}}$ denote the pseudo-distance on $[0,1]$ defined by
\[
d_{\mathbf{Z}}(s,t) = \mathbf{Z}(s) + \mathbf{Z}(t) - 2 \max\left(\min_{x \in [\min(s,t), \max(s,t)]} \mathbf{Z}(x), \min_{x \in [0,1] \, \setminus \, ]\min(s,t), \max(s,t)[ } \mathbf{Z}(x)     \right)
\]
for $s,t \in [0,1]$. Let $\mathbf{D} \in \cC([0,1], \ndR^2)$ denote the largest pseudo-distance on $[0,1]$ that satisfies $\mathbf{D} \le d_{\mathbf{Z}}$ and $\{ d_{\mathbf{e}} = 0\} \subset \{ \mathbf{D} =0 \}$. The Brownian map $(\mathbf{M}, d_{\mathbf{M}})$ may be defined as the quotient  space $[0,1]/\{\mathbf{D}=0\}$ corresponding to $\mathbf{D}$, where any points $s,t$ are identified if $\mathbf{D}(s,t)=0$. The Borel probability measure $\mu_{\mathbf{M}}$ on  $(\mathbf{M}, d_{\mathbf{M}})$ is defined as the push-forward of the Lebesgue measure along the canonical surjection $[0,1] \to \mathbf{M}$.

The measure $\mu_{\mathbf{M}}$ on the Brownian map is almost surely fully supported and diffuse. This enables the application of Lemma~\ref{le:conconlambda} and Lemma~\ref{co:fullrandom} to any sequence of random spaces converging in distribution to the Brownian map.

\begin{comment}
By the main result of~\cite{MR3256874},
\begin{align}
	\label{eq:mapconverge}
	\left(\mM_n, (8n/9)^{-1/4} d_{\mM_n}, \mu_{\mM_n}\right) \convdis (\mathbf{M}, d_{\mathbf{M}}, \mu_{\mathbf{M}})
\end{align}
holds in the Gromov--Hausdorff--Prokhorov sense. To be precise,~\cite{MR3256874} states only Gromov--Hausdorff convergence, but the proof given there also entails Gromov--Hausdorff--Prokhorov convergence: It is shown that there is an ordering  $c(0), c(1), \ldots, c(2n) = c(0)$ of the corners of $\mM_n$ such that a linear interpolation $\tilde{D}_n$ on $[0,1]^2$ of the function
\[
\left\{\frac{0}{2n}, \frac{1}{2n}, \ldots, \frac{2n}{2n}\right\}^2 \to \ndR, \quad \left(\frac{i}{2n},\frac{j}{2n}\right) \mapsto (8n/9)^{-1/4}d_{\mM_n}(c(i), c(j))
\]
satisfies
\begin{align}
	\tilde{D}_n \convdis \mathbf{D}
\end{align}
as random elements of $\cC([0,1]^2, \ndR)$.

If the link weight $\iota$ is supported on a compact subset of $]0, \infty[$, then by the main result of~\cite{2019arXiv190610079L}  there is a constant $\kappa>0$ (that depends on the distribution of $\iota$) such that
\begin{align}
	\label{eq:leherthefpp}
	n^{-1/4} \sup_{x,y \in \mM_n} \left | \kappa d_{\mM_n}(x,y) -  d_{\mathrm{fpp}}(x,y)\right | \convp 0,
\end{align}
and consequently
\begin{align}
	\label{eq:leherfpscal}
	\left(\mM_n, \kappa^{-1} \left(8n/9\right)^{-1/4}  d_{\mathrm{fpp}}, \mu_{\mM_n}\right) \convdis (\mathbf{M}, d_{\mathbf{M}}, \mu_{\mathbf{M}}).
\end{align}
\end{comment}

\section{Simple triangulations and $3$-connected cubic planar graphs}

\subsection{Simple triangulations}

A \emph{planar map} is an embedding of a connected multi-graph onto the $2$-dimensional sphere. Edges of the multi-graph are represented by arcs that may only intersect at their endpoints.  The connected components created when removing a planar map from the $2$-sphere are its \emph{faces}. All faces are required to be homeomorphic to  open discs. 

We view planar maps up to orientation-preserving homeomorphism. This way, there is only a finite number of planar maps with a given number of edges. In order to eliminate symmetries, we distinguish and orient a root edge.

The number of half-edges on the boundary of a face is its \emph{degree}. That is, any edge that has both sides incident to the same face is counted twice. The face to the right of the root edge is called the \emph{outer face}, the one to the left the \emph{root face}. 

A planar map is called a \emph{triangulation}, if all its faces have degree $3$. It is called \emph{simple}, if it has no multi-edges or loops. This way, a triangulation with at least~$4$ vertices is simple if and only if it is $3$-connected. The reader should beware that this convention is not used uniformly throughout the literature. In particular, Tutte~\cite{zbMATH03169204} additionally requires a simple triangulation to have no separating $3$-cycles.

We let $\mQ_n$ denote the uniform simple triangulation with $n+2$ vertices.\footnote{It would seem  natural to use the letter $T$ to denote triangulations instead. However, the author uses the letter $T$ exclusively to refer to trees.} We equip it with the graph distance $d_{\mQ_n}$ and the uniform measure $\mu_{\mQ_n}$ on its set of vertices. The asymptotic shape of $\mQ_n$ was described in~\cite{MR3706731} by establishing the Brownian map as its scaling limit:
\begin{lemma}[{\cite[Thm. 1.1]{MR3706731}}]
	\label{le:GHPconv}
	As $n \to \infty$,
	\[
		(\mQ_n, (3/8)^{1/4} n^{-1/4} d_{\mQ_n}, \mu_{\mQ_n}) \convdis (\mathbf{M}, d_{\mathbf{M}}, \mu_{\mathbf{M}})
	\]
	in the Gromov--Hausdorff--Prokhorov sense.
\end{lemma}

The scaling limit for $\mQ_n$ also holds for the stationary measure. This may be deduced from Lemma~\ref{le:GHPconv} using Corollary~\ref{co:fullrandom} and analogous arguments as for~\cite[Lem. 5.1]{MR3729639}:

\begin{corollary}
	\label{co:coqconv}
Let $\tilde{\mu}_{\mQ_n}$ of $\mQ_n$ the stationary distribution on $\mQ_n$, that describes a random vertex selected with probability proportional to its degree. Then
\begin{align*}
	(\mQ_n, (3/8)^{1/4} n^{-1/4} d_{\mQ_n}, \tilde{\mu}_{\mQ_n}) \convdis (\mathbf{M}, d_{\mathbf{M}}, \mu_{\mathbf{M}})
\end{align*}
as $n \to \infty$, in the Gromov--Hausdorff--Prokhorov sense.
\end{corollary}
\begin{proof}
	Let $Q$ be an   non-empty proper induced subgraph of $\mQ_n$. We say a face of $\mQ_n$ is an internal face of $Q$, if all three vertices incident to the face belong to $Q$. If exactly two of the vertices of the face belong to $Q$, we say it is a boundary face of $Q$. We let $I$ denote the number of internal faces of $Q$.
	
	Let $K \ge 1$ denote the number of connected components of $Q$. Thus, in case $K \ge 2$, $Q$ has a unique face that cannot be homeomorphic to a disc. We  are going to indulge in a slight abuse of notation and nevertheless treat $Q$ as a planar map with the usual concepts of faces and degrees of faces. Specifically, we view $Q$ as a triangulation with boundaries. That is, apart from the internal faces it has a number $H \ge 1$ of ``holes'' which are faces of $Q$ that do not correspond to faces of $\mQ_n$, since $\mQ_n$ has additional vertices inside that face.   We  let $L$ denote the sum of the degrees of the holes.
	
	 We let $V$ denote the total number of vertices of $Q$. By Euler's formula, it holds that
	\[
		V - E + (I+H) = 1 + K.
	\]
	Moreover, since any internal face has degree $3$, 
	\[
		2 E = 3I + L.
	\]
	Hence,
	\begin{align}
		V - (L+I)/2 + H = 1 + K.
	\end{align}

	Let $Q^+$ denote the subgraph induced by all vertices of $\mQ_n$ with graph distance at most $1$ from $Q$. Note that if $f$ is a boundary face of $Q$, then all edges of $f$ lie within~$Q^+$. Furthermore, any boundary face of $Q$ is incident to exactly one edge of $Q$, hence $Q$ has at least $L$ boundary faces.
	
	Let $I^+$ denote the number of faces of $\mQ_n$ that are internal to $Q^+$. It follows that the degrees $d_{\mQ_n}(v)$ of vertices $v \in Q^+$ satisfy
	\begin{align*}
		\sum_{v \in Q^+} d_{\mQ_n}(v) &\ge 3 I^+ \\
		&\ge 3(I + L) \\
		&= 6(V + H - K - 1)  \\
		&\ge 6(V-K).
	\end{align*}
	Note that $\mQ_n$ has $n+2$ vertices and therefore $3n$ edges and $2n$ faces. Consequently,
	\[
		\sum_{v \in Q^+} d_{\mQ_n}(v) = 6n.
	\] It follows that
	\begin{align}
		\label{eq:arbitraryQ}
		\tilde{\mu}_{\mQ_n} (Q^+) \ge \frac{V-K}{n} \ge \mu_{\mQ_n}(Q) - \frac{K}{n}.
	\end{align}
	Note that this inequality also holds in the trivial cases when $Q$ is empty or $Q = \mQ_n$. 

	Let $\epsilon, \epsilon'>0$  be given. Since the Brownian map is almost surely diffuse, it follows by Lemma~\ref{le:GHPconv} and  Corollary~\ref{co:fullrandom} that there exists $\delta>0$ such that
	\begin{align}
		\Prb{\inf_{x \in \mQ_n} \mu_{\mQ_n} (C_{\epsilon n^{1/4}}^{Q_n}(x)) \ge \delta} \ge 1- \epsilon'
	\end{align}
	for all large enough $n$. Now, if we are in the case
	\[
	\inf_{x \in \mQ_n} \mu_{\mQ_n} (C_{\epsilon n^{1/4}}^{Q_n}(x)) \ge \delta,
	\]
	then each connected component of the $2\epsilon$-hull $U_{2 \epsilon n^{1/4}}^{\mQ_n}(Q)$ contains a closed $\epsilon$-ball and has hence at least $\delta(n+2)$ vertices. Since Inequality~\eqref{eq:arbitraryQ} holds for arbitrary subsets, it follows that
	\begin{align*}
		\tilde{\mu}_{\mQ_n} ((U_{2 \epsilon n^{1/4}}^{\mQ_n}(Q) )^+) &\ge \mu_{\mQ_n}(U_{2 \epsilon n^{1/4}}^{\mQ_n}(Q)) - \frac{1}{\delta n} \\
		&\ge \mu_{\mQ_n}(Q) - \frac{1}{\delta n}.
	\end{align*}
Thus
	\begin{align*}
		\mu_{\mQ_n}(Q)  \le \frac{1}{\delta n} + \tilde{\mu}_{\mQ_n} \left(U_{2 \epsilon n^{1/4}+2}^{\mQ_n}(Q) \right).
	\end{align*}
Since $\epsilon$ and $\epsilon'$ where arbitrary, it follows that
\begin{align}
	d_{\mathrm{P}}(\mu_{\mQ_n}, \tilde{\mu}_{\mQ_n} )\convp 0
\end{align}
when considering ${\mu}_{\mQ_n}$ and $\tilde{\mu}_{\mQ_n}$ as measures on the rescaled triangulation $(\mQ_n, (3/8)^{1/4} n^{-1/4} d_{\mQ_n})$.
By Lemma~\ref{le:GHPconv} it follows that
\begin{align*}
	(\mQ_n, (3/8)^{1/4} n^{-1/4} d_{\mQ_n}, \tilde{\mu}_{\mQ_n}) \convdis (\mathbf{M}, d_{\mathbf{M}}, \mu_{\mathbf{M}}).
\end{align*}
\end{proof}

The framework in~\cite{MR3706731} also allows us to deduce a deviation bound for the diameter~$\Di(\mQ_n)$: 

\begin{proposition}
	\label{pro:diamtria}
	For any $\epsilon>0$ there are constants $C,c$ such that for all integers $n \ge 2$
	\[
		\Pr{\Di(\mQ_n) > n^{1/4 + \epsilon}} < C \exp(-cn^{\epsilon}).
	\]
\end{proposition}
\begin{proof}
	Let $\xi$ denote the random non-negative integer with distribution
	\[
		\Pr{\xi = k} = \frac{27}{128} (k+1)(k+2)(1/4)^k.
	\]
	Note that $\Ex{\xi}=1$. Let $\mT_n$ denote a $\xi$--Bienaym\'e--Galton--Watson tree conditioned on having $n$ vertices.  By the tail-bounds~\cite[Thm. 1.2]{MR3077536}, it follows that the height $\He(\mT_n)$ satisfies for all $x \ge 0$
	\begin{align}
		\label{eq:comb1}
		\Pr{\He(\mT_n) \ge x} \le C_0 \exp( - c_0 x^2 / n)
	\end{align}
	for some constants $C_0, c_0>0$ that do not depend on $n$ or $x$.

	For any vertex $v$ of $\mT_n$ with children $v_1, \ldots, v_d$ we add random labels $D_1, \ldots, D_d$ from $\{-1, 0, 1\}$ to the edges between $v$ and its children in a uniformly selected manner so that $(D_i)_{1 \le i \le d}$ is non-decreasing.
	
	After assigning labels to all edges of $\mT_n$ in this way, we assign an integer label $Y(v)$ to each vertex $v$ of $\mT_n$ given by the sum of $2$ and the labels on the edges from the unique path between $v$ and the root of $\mT_n$. 
	
	As described in~\cite[Sec. 5]{MR3706731}, the random $n+2$-vertex triangulation may be constructed from the decorated tree $(\mT_n, (Y(v))_{v \in \mT_n})$ by adding two vertices $A,B$ and a number of edges so that $A$ and $B$ are neighbours. This is done in a way, so that by~\cite[Cor. 7.5, Fact 5.9]{MR3706731}
	\begin{align}
		\label{eq:comb2}
		\sup_{ u \in \mT_n} d_{\mQ_n}(u, A) \le 2 \sup_{u \in \mT_n} |Y(u)| + O(1).
	\end{align}
	
	Let $u_n$ denote a random vertex of $\mT_n$. Conditional on having a height $\he_{\mT_n}(u_n) = h$ for $h \ge 0$, the labels along the path from the root to $u_n$ are that of a centred $h$-step random walk with increments in $\{-1,0,1\}$.  By the Azuma--Hoeffding inequality, it follows that
	\begin{align}
		\Prb{|Y(u_n)| \ge x  \mid \he_{\mT_n}(u_n) = h} \le 2 \exp( -x^2 / (2h)).
	\end{align}
	Hence
	\begin{align}
		\label{eq:comb3}
		\Prb{ \sup_{u \in \mT_n} |Y(u)| \ge x \,\, \Big\vert \,\, \He(\mT_n)  } \le 2 n \exp\left( -\frac{x^2}{2\He(\mT_n)} \right).
	\end{align}
	Combining~\eqref{eq:comb1},~\eqref{eq:comb2}, and~\eqref{eq:comb3} yields
	\begin{align*}
		\Pr{ \Di(\mQ_n) \ge n^{1/4 + \epsilon} } &\le C_0 \exp(-c_0 n^{\epsilon}) + \Pr{\Di(\mQ_n) \ge n^{1/4 + \epsilon} , \He(\mT_n) \le n^{1/2 + \epsilon/2}} \\
		&\le C_0 \exp(-c_0 n^{\epsilon}) + 2n \exp\left(- n^{3 \epsilon /2 + o(1)} \right) \\
		&\le C \exp(-c n^{\epsilon})
	\end{align*}
for some constants $C,c>0$ that do not depend on $n$.
\end{proof}

\subsection{The dual map construction}

The \emph{dual map} $Q^\dagger$ of a triangulation $Q$ is the  ``red'' cubic planar map constructed by placing a red vertex inside each face of $Q$ and then adding for each edge $e$ of $Q$ a red edge between the red vertices inside the two faces adjacent to $e$. These faces may be identical, and in this case the corresponding red edge is a loop. The root edge of the dual map $Q^\dagger$ is the red edge corresponding to the root-edge of $Q$, oriented in a canonical way.
For any vertex $u \in Q$ and any vertex $f \in Q^\dagger$ we write $u \triangleleft f$ if $u$ is adjacent to the face of $Q$ that corresponds to $f$.

The graph distance  $d_Q$ in $Q$ and the graph distance $d_{Q^\dagger}$ in the dual $Q^\dagger$ satisfy crude bounds:
\begin{proposition}
	\label{pro:roughbound}
	Let $u, v \in Q$ denote vertices in a triangulation $Q$ and let $f,g \in Q^\dagger$ be points in the dual with $u \triangleleft f$ and $v \triangleleft g$. Let $\Delta(Q) = \max_{v \in Q} d_Q(v)$ denote the largest among the degrees $d_Q(v)$ of vertices $v \in Q$.  Then
	\[
		d_Q(u,v) \le d_{Q^\dagger}(f,g) +1
	\]
	and
	\[
		d_{Q^\dagger}(f,g) \le (d_Q(u,v)+1) \Delta(Q).
	\]
\end{proposition}
\begin{proof}
	We start with the first inequality.
	If $f = g$, then clearly this is true. 
	If $f_1 \ne f_2$ are faces of $Q$ sharing an edge $e$, then any vertex on the boundary of $f_1$ has graph distance at most $1$ from the edge $e$. Hence if $f_1, \ldots, f_k$ is a sequence of faces such that $f_i$ and $f_{i+1}$ share an edge $e_{i+1}$ for each $1 \le i \le k-1$, then starting from any vertex on the boundary of $f_1$ we may reach the edge $e_{k}$ by a path of length at most $k-1$. Any vertex on the boundary of $f_k$ has graph distance at most $1$ from $e_k$, showing that arbitrary pairs of points from the boundaries of $f_1$ and $f_k$ have graph distance at most $k$. This proves the first inequality.
	
	As for the second inequality, let $v_1, \ldots, v_k$ be a path in $Q$ and let $f,g \in Q^\dagger$ with $v_1 \triangleleft f$ and $v_k \triangleleft g$. The edge $v_1 v_2$ is incident to a face $f_1 \in Q^\dagger$ that we can reach from $f$ in $Q^\dagger$ by traversing the faces incident to $v_1$ in a circular manner. This takes at most $d_Q(v_1)-1$ steps. Proceeding in this way, we can walk from $g$ to  a face incident to the edge $v_{k-1}v_k$ and from there to the face $g$ by a path in $Q^\dagger$ with length at most
	\[
		\sum_{i=1}^{k} (d_Q(v_i) -1) \le k \Delta(Q).
	\]
	This proves the second inequality.
\end{proof}
This crude bound suffices to get a deviation bound for the dual:
\begin{proposition}
	\label{pro:produalbound}
	For any $\epsilon>0$ there are constants $C,c>0$ such that for all integers $n \ge 2$
	\[
	\Pr{\Di(\mQ_n^\dagger) > n^{1/4 + \epsilon}} < C \exp(-cn^{\epsilon/2}).
	\]
\end{proposition}
\begin{proof}
	As shown by~\cite{zbMATH00683269} and~\cite[Lem. 4.1]{MR2013797}, there are constants $C_0, c_0>0$ such that origin $v$ of the root edge of  $\mQ_n$ satisfies
	\begin{align}
		\Pr{d_{\mQ_n}(v) \ge k} \le C_0 \exp(-c_0k)
	\end{align}
for all $n \ge 2$, $ k \ge 0$. The distribution of $\mQ_n$ is invariant under re-rooting at a uniformly selected oriented edge. As $\mQ_n$ has $n+2$ vertices and hence $3n$ edges, it follows that the maximum degree $\Delta(\mQ_n)$ satisfies
	\begin{align*}
		\Pr{\Delta(\mQ_n) \ge k} &\le 6n  \Pr{d_{\mQ_n}(v) \ge  k} \\
		&\le 6C_0n \exp(-c_0 k).
	\end{align*}
Hence there are constants $C_1,c_1>0$ that do not depend on $n$ such that
\begin{align}
	\label{eq:diamcomba}
\Pr{\Delta(\mQ_n) \ge n^{\epsilon/2}} < C_1 \exp(-c_1 n^{\epsilon/2}).
\end{align}
Moreover, by Proposition~\ref{pro:diamtria} there are constants $C_2, c_2>0$ such that
\begin{align}
	\label{eq:diamcombb}
\Pr{\Di(\mQ_n) \ge n^{1/4 + \epsilon/2}} < C_2 \exp(-c_2n^{\epsilon/2}).
\end{align}
Using Proposition~\ref{pro:roughbound} and applying~\eqref{eq:diamcomba} and~\eqref{eq:diamcombb}, it follows that
\begin{align*}
	\Pr{\Di(\mQ_n^\dagger) > n^{1/4 + \epsilon}} &\le \Pr{ \Di(\mQ_n) \Delta(\mQ_n) > n^{1/4 + \epsilon}} \\
	&\le C_3 \exp(-c_3 n^{\epsilon/2})
\end{align*}
for $C_3 = C_1 + C_2$ and $c_3 = \min(c_1,c_2)$.
\end{proof}

Let $\iota>0$ denote a positive random variable. The $\iota$-first-passage percolation metric on a planar map assigns a weight to each edge according to an independent copy of $\iota$. The distance between any two points is then given by the minimal sum of weights along joining paths.

We define $d^\dagger_{\mathrm{fpp}}$ as the $\iota$-first-passage percolation distance on the dual $\mQ_n^\dagger$ of the random simple triangulation $\mQ_n$. It concentrates around a constant multiple of the graph distance on $\mQ_n$:

%The first passage percolation distance $d^\dagger_{\mathrm{fpp}}$ for a random simple triangulation $\mQ_n$ with $n+2$ vertices concentrate at a constant multiple of the graph distance. 

\begin{lemma}[{\cite[Thm. 3.1]{simplefpp}}]
	\label{le:left}
	Let $\iota>0$ denote a random variable with finite exponential moments that satisfies $\Pr{\iota\ge c}=1$ for some constant $c>0$. Then there exists a constant $c_E>0$ such that
	\begin{align}
		\label{eq:thefpp}
		n^{-1/4} \sup_{\substack{u,v \in \mQ_n\\ u\triangleleft f, v\triangleleft g }} \left | c_E d_{\mQ_n}(u,v) -  d_{\mathrm{fpp}}^\dagger(f,g)\right | \convp 0.
	\end{align}
\end{lemma}

Such a statement was proven in~\cite[Thm. 3]{zbMATH07144469} for uniform unrestricted triangulations with $\iota$-weights following an exponential distribution, and $\iota$-weights that are constant, so that $d_{\mathrm{fpp}}^\dagger$ corresponds to a multiple of the dual map distance. The proof of this result in~\cite{simplefpp} is by extending the arguments for~\cite[Thm. 3]{zbMATH07144469}.

\subsection{$3$-connected cubic planar graphs}

 Restricting the dual map construction to simple triangulations with at least $4$ vertices yields a bijection between such objects and  $3$-connected cubic planar maps. By Whitney's theorem, any $3$-connected  planar graph has precisely two embeddings in the plane, and one is the mirror image of the other. Thus, there is a two-to-one correspondence between $3$-connected cubic planar maps and $3$-connected cubic planar graphs with an oriented root edge.

We let
\begin{align}
	\cQ(z) = \sum_{n \ge 1} q_n z^n
\end{align}
denote the generating series for simple triangulations, with $q_n$  denoting the number of simple triangulations with $n+2$ vertices (and hence $2n$ faces and $3n$ edges). Specifically, as determined by Tutte~\cite{TUTTE1973437, zbMATH03169204}, 
\begin{align}
	\label{eq:enumsimple}
	q_n = \frac{\sqrt{6}}{32 \sqrt{\pi}} n^{-5/2} \left( \frac{27}{256} \right)^{-n} \left(1 + O\left(\frac{1}{n}\right) \right). 
\end{align}
It follows that
\begin{align}
	\label{eq:mq}
	\cM(x,y) = \frac{1}{2}( \cQ(x^2y^3) - x^2y^3 )
\end{align}
is the  generating series for  $3$-connected cubic planar graphs with an oriented root edge, with $x$ marking vertices and $y$ marking edges. For any even integer $n \ge 4$ we let $\mM_n$ denote the uniform $n$-vertex $3$-connected cubic planar graph. Thus, by Whitney's theorem,  $\mM_n$ is distributed like the dual of the random simple triangulation $\mQ_{n/2}$. Hence, Proposition~\ref{pro:produalbound} readily yields:
\begin{corollary}
	\label{cor:3condeviation}
	For any $\epsilon>0$ there are constants $C,c>0$ such that for all integers $n \ge 2$
	\[
	\Pr{\Di(\mM_n) > n^{1/4 + \epsilon}} < C \exp(-cn^{\epsilon/2}).
	\]
\end{corollary}

 We let $d_{\mM_n}$ denote the graph distance on $\mM_n$ and $\mu_{\mM_n}$ the uniform measure on the set of vertices of $\mM_n$. 

\begin{theorem}
	\label{te:cubic3con}
	There exists a constant $c^\dagger>0$ such that
	\begin{align*}
		\left(\mM_n, \frac{3^{1/4}}{2 c^\dagger} n^{-1/4} d_{\mM_n}, {\mu}_{\mM_n}\right) \convdis (\mathbf{M}, d_{\mathbf{M}}, \mu_{\mathbf{M}})
	\end{align*}
	in the Gromov--Hausdorff--Prokhorov sense as $n \in 2\ndN$ tends to infinity.
\end{theorem}
\begin{proof}
By Lemma~\ref{le:left} there exists a constant $c^\dagger>0$ with
\begin{align}
	\label{eq:opnquarter}
 	n^{-1/4} \sup_{\substack{u,v \in \mQ_n\\ u\triangleleft f, v\triangleleft g }} \left | c^\dagger d_{\mQ_{n/2}}(u,v) -  d_{\mQ_{n/2}^\dagger}(f,g)\right | \convp 0.
\end{align}
The relation $\triangleleft$ is yields a correspondence between the vertex set of the rescaled triangulation \[
(\mQ_{n/2}, (3^{1/4}/2) n^{-1/4} d_{\mQ_{n/2}}, \tilde{\mu}_{\mQ_{n/2}})
\]
and the vertex set of the rescaled dual 
\[
(\mQ_{n/2}^\dagger, (3^{1/4}/(2c^\dagger)) n^{-1/4} d_{\mQ_{n/2}^\dagger}, {\mu}_{\mQ_{n/2}^\dagger}),\]
with ${\mu}_{\mQ_{n/2}^\dagger}$ denoting the uniform distribution on $\mQ_{n/2}^\dagger$
. By~\eqref{eq:opnquarter}, it follows that
\begin{align}
	\mathrm{dis}(R) = o_p(1).
\end{align}
We select a point $f \in \mQ_{n/2}^\dagger$ uniformly at random, and then select a vertex $u$ with $u \triangleleft f$ uniformly at random. This way, $u$ assumes a vertex of  $\mQ_{n/2}$ with probability proportional to the number of faces on whose boundary it lies. Since $\mQ_{n/2}$ is a simple triangulation, this  is precisely the degree of $u$. Hence $u$ follows the stationary distribution $ \tilde{\mu}_{\mQ_{n/2}}$. Thus, the joint law of $u$ and $f$ is a coupling \[
\nu \in M\left( \tilde{\mu}_{\mQ_{n/2}}, \mu_{\mQ_{n/2}^\dagger}\right)
\]
 between $\tilde{\mu}_{\mQ_{n/2}}$ and  $\mu_{\mQ_{n/2}^\dagger}$ . By construction,\[
 \nu(R) = 1.
 \]
Using Proposition~\ref{pro:ghpcoupling}, it follows that
\begin{align}
	\label{eq:tmp1234}
d_{\mathrm{GHP}}\left( \left(\mQ_{n/2}, \frac{3^{1/4}}{2} n^{-1/4} d_{\mQ_{n/2}}, \tilde{\mu}_{\mQ_{n/2}}\right) , \left(\mQ_{n/2}^\dagger, \frac{3^{1/4}}{2c^\dagger} n^{-1/4} d_{\mQ_{n/2}^\dagger}, {\mu}_{\mQ_{n/2}^\dagger}\right)\right) \convp 0.
\end{align}
By Corollary~\ref{co:coqconv}, we know that
\begin{align}
	\left(\mQ_{n/2}, \frac{3^{1/4}}{2} n^{-1/4} d_{\mQ_{n/2}}, \tilde{\mu}_{\mQ_{n/2}}\right) \convdis (\mathbf{M}, d_{\mathbf{M}}, \mu_{\mathbf{M}})
\end{align}
$n \in 2\ndN$ tends to infinity. Combining this with~\eqref{eq:tmp1234}, we obtain
\begin{align}
		\left(\mQ_{n/2}^\dagger, \frac{3^{1/4}}{2c^\dagger} n^{-1/4} d_{\mQ_{n/2}^\dagger}, {\mu}_{\mQ_{n/2}^\dagger}\right) \convdis (\mathbf{M}, d_{\mathbf{M}}, \mu_{\mathbf{M}}).
\end{align}
As discussed above, Whitney's theorem ensures that   $\mM_n$ is distributed like the dual of the random simple triangulation $\mQ_{n/2}$. Hence this completes the proof.
\end{proof}

Given a random link-weight $\iota>0$, we let $d_{\mathrm{fpp}}$ denote the corresponding first-passage percolation distance on $\mM_n$.
Applying Lemma~\ref{le:left} twice immediately yields the following concentration result:
\begin{lemma}[{\cite[Thm. 1.2]{simplefpp}}]
	\label{le:fpp3con}
	Let $\iota>0$ denote a 	random positive integer with finite exponential moments. There exists a constant $c_{\mathrm{fpp}}>0$ such that
	\begin{align}
		\label{eq:thefpponcore}
		n^{-1/4} \sup_{u,v \in \mM_n} \left |c_{\mathrm{fpp}} d_{\mM_n}(u,v) -  d_{\mathrm{fpp}}(u,v)\right | \convp 0
	\end{align}
as $n \in 2\ndN$ tends to infinity.
\end{lemma}

\section{The scaling limit of cubic planar graphs}
\label{sec:cuplag}

\subsection{Network decomposition}
\label{sec:netdec}

We recall the network decomposition~\cite{zbMATH05122852, zbMATH07213288} of cubic planar graphs. 
\begin{definition}
A \emph{(cubic) network} is a connected planar cubic multi-graph $N$ with an oriented root edge $e$ such that the graph $N- e$ obtained by removing $e$ is simple.
\end{definition}

The endpoints of the root edge are the \emph{poles} of the network. The class $\cN$ of networks has an exponential generating series $\cN(x)$. For all  $n \ge 0$ the $n$th coefficient $[x^n]\cN(x)$ is given by $1/n!$ times the number of networks with vertices labelled from $1$ to $n$. This coefficient is equal to zero unless $n \in \{2i \mid i \ge 2\}$.

\begin{figure}[t]
	\centering
	\begin{minipage}{\textwidth}
		\centering
		\includegraphics[scale=0.5]{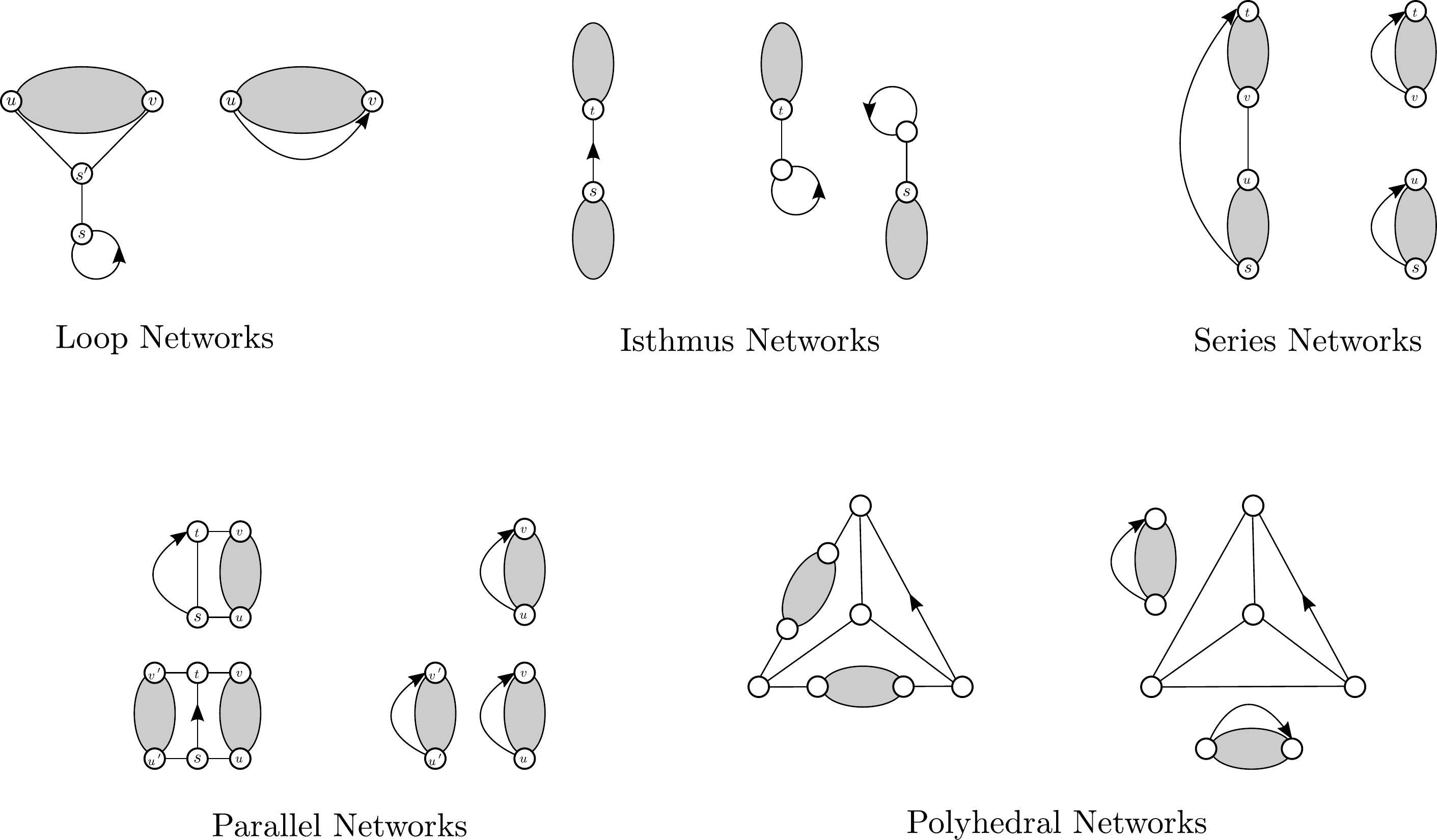}
		\caption{Decomposition of networks.}
		\label{fi:total}
	\end{minipage}
	%  \caption[Decomposition of networks.]{Decomposition of networks.}
\end{figure}

As argued in~\cite{zbMATH05122852}, networks may be divided into five subclasses:
\begin{enumerate}[\qquad 1)]
	\item  $\cL$ (Loop). The root edge $e$ is a loop.
	\item $\mathcal{I}$ (Isthmus). The root edge $e$ is an isthmus (also called a bridge). That is, the network $N - e$ obtained by removing $e$ is disconnected.
	\item $\mathcal{S}$ (Series). $N - e$ is connected, but contains a bridge that separates the endpoints of $e$.
	\item $\mathcal{P}$ (Parallel). $N - e$ is connected, contains no bridge that would separate the endpoints of $e$, and either $e$ is part of a double edge in $N$ or deleting the endpoints of $e$ disconnects $N$.
	\item $\mathcal{H}$ (Polyhedral). $N$ is obtained from a $3$-connected network by possibly replacing each non-root edge with a non-isthmus network.
\end{enumerate}

We define the subclass $\cD$ of non-isthmus networks so that
\begin{align}
	\label{eq:n}
	\cN(x) &= \cD(x) + \cI(x), \\
	\label{eq:d}
	\cD(x) &= \cL(x) + \cS(x) + \cP(x) + \cH(x).
\end{align}
As illustrated in Figure~\ref{fi:total}, each of the five classes may be decomposed in terms of the other classes and itself. yielding, as described in~\cite{zbMATH05122852}, the equation system
\begin{align}
	\label{eq:l}
	\cL(x) &= \frac{x^2}{2}(\cN(x) - \cL(x)), \\
	\label{eq:i}
	\cI(x) &= \frac{\cL^2}{x^2}, \\
	\label{eq:s}
	\cS(x) &= (\cN(x)- \cS(x) - \cI(x))(\cN(x) - \cI(x)), \\
	\label{eq:p}
	\cP(x) &= x^2 \cD(x) + \frac{x^2}{2}\cD(x)^2, \\
	\label{eq:h}
	\cH(x) &= \frac{\cM(x, 1 + \cD(x))}{1 + \cD(x)}.
\end{align}

We briefly recall the details of these decompositions.
\subsubsection{Loop networks}

The network $N$ is a loop network, if its poles $s$ and $t$ are identical. As illustrated in Figure~\ref{fi:total}, in this case the vertex $s=t$ is adjacent to a unique vertex $s'$, which is adjacent to two vertices $u \ne v$. Deleting $s, s'$  and adding a root edge from $u$ to $v$ or from $v$ to $u$ yields a non-loop network. 

\subsubsection{Isthmus networks}

As depicted in Figure~\ref{fi:total}, an isthmus network may be decomposed into an ordered pair of loop networks, each having an additional vertex. 

\subsubsection{Series networks}

Given a series network $N$ with root edge $e$, the graph $N -e$ contains bridges that separate the poles $s$ and $t$. Among those bridges we may select the one that is closest to $s$. Deleting it creates two components, one containing $s$ and an end $u$ of the bridge, the other containing $t$ and the other end $v$ of the bridge. Thus, adding directed edges $su$ and $vt$ creates two networks $N_1$ and $N_2$. Neither can be an isthmus network. Since we chose the bridge closest to $u$, the network~$N_1$ cannot be a series network.

\subsubsection{Parallel networks}

We distinguish two types of parallel networks. In the first type, the root is a double edge. Hence the poles are adjacent to vertices $u$ and $v$. Deleting $s$ and $t$ and adding a root edge from $u$ to $v$ yields a non-isthmus network. In the second type of parallel networks the root is not a double edge. As illustrated in Figure~\ref{fi:total}, such networks correspond to unordered pairs of non-isthmus networks.

\subsubsection{Polyhedral networks}

A polyhedral network is constructed  from a $3$-connected cubic planar graph $M$ with an oriented root edge by inserting components $D(1), D(2), \ldots$ at its non-root edges. To do so, we order and direct the non-root edges of $M$ in a canonical way. Each component is either a non-isthmus network or a placeholder value. For each list index $i$ the $i$th edge in the list is either left as is if $D(i)$ is equal to the placeholder value, or otherwise replaced by $D(i)$ as illustrated in Figure~\ref{fi:total}.

\subsection{Distances between poles of networks}

In a polyhedral network, inserting a network  component $D$ at an edge $e$ of the $3$-connected core means deleting the edge $e$ and the root-edge of $D$, and adding two edges to connect the poles of $D$ to the former endpoints of $e$. See Figure~\ref{fi:total} for an illustration.

For this reason, we need to understand the graph distance between the poles of a non-isthmus network after removing the root edge. The goal in this section is to argue that in a ``free'' non-isthmus network, this distance has finite exponential moments.

We start by defining such a random network model. As shown in~\cite[Proof of Thm. 1]{zbMATH07213288} the  generating series $\cD(x)$ has an algebraic radius of convergence $\rho>0$ satisfying
\begin{align*}
	\rho &=       0.319224606195452700761429068280\ldots, \\
	\cD(\rho) &=  0.011525944379127380775581944095\ldots,
\end{align*}
and
\begin{align}
	\label{eq:rod}
	\rho^2(1 + \cD(\rho))^3 = \frac{27}{256}.
\end{align}
The points $\{\rho, - \rho\}$ are the only singularities on the circle $\{z \in \ndC \mid |z|=\rho\}$, and as $x \to \rho$
\begin{align}
	\label{eq:singhere}
	\cD(x) = D_0 + D_2 \left(1 - \frac{x}{\rho}\right) + D_{3} \left(1 - \frac{x}{\rho}\right)^{3/2} + O\left( \left(1 - \frac{x}{\rho}\right)^2 \right)
\end{align}
for  constants $D_0 = \cD(\rho)$, $D_2 = - \rho \cD'(\rho)$, $D_3>0$. As shown in~\cite[Proof of Thm.1]{zbMATH07213288},  each class $\cF \in \{\cL,  \cS, \cP, \cH\}$ may be expressed in terms of $\cD$, in particular
\[
	\cL(x) = 1 + \frac{x^2}{2} - \sqrt{\frac{x^2}{4} + 1 -x^2(\cD(x) -1)}
\] 
and
\[
	\cP(x) = \frac{\cD^2(x)}{1+ \cD(x)}.
\]
The constants $\rho$ and $\cD(\rho)$ may be determined algebraically by solving  the resulting equation 
\begin{align}
	\cD = 1 + \frac{x^2}{2} - \sqrt{\frac{x^2}{4} + 1 -x^2(\cD -1)} + \frac{\cD^2}{1+ \cD} + x^2 \cD + \frac{x^2}{2}\cD^2 +  \frac{\cM(x, 1 + \cD)}{1 + \cD}.
\end{align}
 simultaneously with Equation~\eqref{eq:rod} for $(\rho, \cD(\rho))$. The singularity expansion~\eqref{eq:singhere} implies
\begin{align}
	\label{eq:asymphere}
	[x^n] \cD(x) \sim c_{\cD} n^{-5/2} \rho^{-n}, \qquad c_{\cD} = \frac{3 D_3}{2 \sqrt{\pi}}
\end{align}
as $n \in 2\ndN$ tends to infinity. Furthermore, since any class $\cF \in \{\cL, \cI, \cS, \cP, \cH\}$ may be expressed in terms of $\cD$, it follows that there exists  $c_{\cF}>0$ with
\begin{align}
	\label{eq:genasymphere}
	[x^n] \cF(x) \sim c_{\cF} n^{-5/2} \rho^{-n}.
\end{align}
Of course, $c_\cF$ depends on the class under consideration.
 Let $Y \ge 0$ denote a random non-negative integer with probability generating function
\begin{align}
	\Ex{x^{Y}} = \frac{1 + \cD(\rho x)}{1 + \cD(\rho)}.
\end{align}
Hence
\begin{align}
	\label{eq:yasmp}
	\Pr{Y=n} \sim \frac{c_{\cD}}{1 + \cD(\rho)} n^{-5/2}
\end{align}
as $n \in 2 \ndN$ tends to infinity.

\begin{definition}[Boltzmann  network]
	\label{de:defd}
	We let $\mD$ denote a random network constructed as follows. If $Y>0$, we set $\mD$ to a uniformly selected non-isthmus network with $Y$ labelled vertices. If $Y=0$, we set $\mD$ to a placeholder value  representing that inserting $\mD$ at an edge of a cubic planar graph leaves the graph unchanged. We refer to $\mD$ as the Boltzmann network. 
\end{definition}

Inserting a copy of $\mD$ at an edge $e$ may increase the length of a shortest path between the endpoints of $e$ that is required to stay in this component. We let $\iota$ denote the length of such a shortest path:

\begin{definition}[Link weight]
	\label{de:linkweight}
	We define a random positive integer $\iota>0$ as follows. If $Y>0$, we set $\iota$ to the sum of $2$ plus the graph distance between the poles of $\mD$ after removing the root-edge. If $Y=0$, we set $\iota=1$. We refer to $\iota$ as the link weight.
\end{definition}

The number of vertices $Y$ of the Boltzmann network $\mD$ has a density that varies regularly with index $-5/2$ by Equation~\eqref{eq:yasmp}. However, the distance $\iota$ has a light tail:

\begin{lemma}
	\label{le:linkw}
	The link weight $\iota$ has finite exponential moments.
\end{lemma}
\begin{proof}
We are going to construct a suitable light-tailed upper bound. To this end, we are going to form a bivariate generating series $\cD(x,w)$, and similar bivariate series for subclasses of non-isthmus networks, with $x$ marking vertices and $w$ marking the length of a specific path between the poles that avoids the root-edge. Thus, the marked path length is not necessarily the length of a geodesic avoiding the root-edge, but it is always an upper bound.

We start with parallel networks. If a parallel network is not simple (this case corresponds to the summand $x^2 \cD$ in Equation~\eqref{eq:p}), then removing the poles still have distance $1$ after removing the root-edge. If a parallel network is not simple (the summand $(x^2/2) \cD^2$ in Equation~\eqref{eq:p}), then it consists of two components, and the distance between the poles after removing the root edge is equal to $2$ plus the minimum of the two corresponding distances in the components. However, for an upper bound we may simply bound the minimum by the length of a specified path in one of the components. Hence,
\begin{align}
	\cP(x,w) = w x^2 \cD(x) + w^2(x^2/2) \cD(x)\cD(x,w).
\end{align} 
For loop networks, the two poles are identical. We mark this path of length zero. This results in
\begin{align}
	\cL(x,w) = \cL(x).
\end{align}
In a series network, the distance between the poles after removing the root edge may be bounded by one plus the sum of the specified path lengths in the two components. Hence,
\begin{align}
	\cS(x,w) = w (\cL(x) + \cP(x,w) + \cH(x,w))\cD(x,w).
\end{align}
For polyhedral networks, recall that $\cM(x,y)$ denotes the exponential generating series of $3$-connected cubic planar graphs with an oriented root edge. Recall that by~\eqref{eq:mq} and Whitney's theorem, \[
\cM^*(x,y) := 2\cM(x,y) = \cQ(x^2y^3) - x^2y^3
\] is the generating series for $3$-connected cubic planar maps, with $x$ marking vertices and $y$ marking edges. The two maps corresponding to such a graph are mirror images of each other, in particular the  root face and outer face are reversed. We extend $\cM^*(x,y)$ to $\cM^*(x,y,v)$ so that $v$ marks the degree of the root face. Thus, \[
\cF(x,y,v) := \frac{1}{2yv}\cM^*(x,y,v)
\] is the generating series of $3$-connected cubic planar graphs equipped with an oriented root edge, with $x$ marking vertices, $y$ marking \emph{non-root} edges, and $v$ marking the length of a path between the poles that avoids the root-edge. Hence, in a polyhedral network obtained by blowing up the non-root edges of the $3$-connected core by network components, we can mark a path between the poles given by the concatenation of the marked components inserted into this marked path in the core. Thus, using~\eqref{eq:h}, we obtain
\begin{align}
	\cH(x,w) = \cF\left(x,1 + \cD(x), \frac{w + w^2 \cD(x,w)}{1 + \cD(x)}\right)
\end{align} 
By~\cite{35bd2f3db4e1480c838f45fa2efcec35} (see also~\cite{zbMATH03732085,zbMATH05641325}), for $n, k \ge 1$ 
\[
	a_{n,k} = \frac{2k(4n-1-2k)!}{(n-k)!(3n-k+1)!} \binom{2k+1}{k}
\]
counts the number of $3$-connected cubic planar maps with $3n + 3$ edges and root face degree $k+2$. The generating series 
\begin{align}
a(z,v) &= 1 + \sum_{n, k \ge 1} a_{n,k} z^n v^k \\
&= 1+v x+v (2 v+1) x^2+v \left(5 v^2+5 v+3\right) x^3+\ldots \nonumber
\end{align}
 was determined in~\cite[Eq. (3), (4)]{35bd2f3db4e1480c838f45fa2efcec35} to equal
\[
	\frac{(1 + u(x))^2}{2 u(x) v^2}\left(v + 3u(x)v - (1 + u(x))^2+ (1 + u(x))(1 + u(x) - v)\sqrt{1 - \frac{4 u(x)v}{(1 + u(x))^2}} \right)
\]
with \[
u(x) = x+4 x^2+22 x^3+140 x^4+969 x^5+7084x^6 + \ldots
\] denoting the unique power series with non-negative coefficients satisfying
\[
	u(x) = x(u(x)+1)^4.
\]
Taking into account the shift of indices, we obtain
\[
	\cM^*(x,y,v) = {v^2}{x^2y^3} (a(x^2 y^3, v)-1).
\]
Hence
\begin{align}
	\label{eq:hw}
	\cH(x,w) = \frac{1}{2}{\frac{w + w^2 \cD(x,w)}{(1 + \cD(x))^2}}{x^2(1 + \cD(x))^3} a\left(x^2 (1 + \cD(x))^3, \frac{w + w^2 \cD(x,w)}{1 + \cD(x)}\right).
\end{align}
Summarizing, we obtain
\begin{align}
	\cD(x,w) =& w x^2 \cD(x) + w^2(x^2/2) \cD(x)\cD(x,w) + \cL(x) + \cH(x,w)  \\
		&+w (\cL(x) +  w x^2 \cD(x) + w^2(x^2/2) \cD(x)\cD(x,w) + \cH(x,w))\cD(x,w). \nonumber
\end{align}
The term $\cL(x)$ may be expressed in terms of $\cD(x)$. Indeed, combining Equations~\eqref{eq:n},~\eqref{eq:l}, and~\eqref{eq:i} we obtain
\[
\cL(x) = \frac{x^2}{2}(\cD(x) + \cI(x) - \cL(x)) = \frac{x^2}{2}(\cD(x) + \cL(x)^2 / x^2 - \cL(x)).
\]
Solving for $\cL(x)$ yields
\begin{align}
	\label{eq:lplugme}
		\cL(x)= 1 + \frac{x^2}{2} - \sqrt{\frac{x^2}{4} + 1 -x^2(\cD(x) -1)}.
\end{align}

Our goal is now to show that the series $\cD(\rho,w)$ has radius of convergence strictly larger than $1$. We start by simplifying the term $\cH(\rho,w)$. Using Equation~\eqref{eq:rod}, we obtain
\[
	u(\rho^2(1 +  \cD(\rho))^3) =  u(27/256) = 1/3.
\]
Thus
\[
a(\rho^2(1 +  \cD(\rho))^3, v) = \frac{16 \left((4-3 v)^{3/2}+9 v-8\right)}{27 v^2}
\]
and

\begin{multline}
	\label{eq:hwrho}
	\cH(\rho, w) = \frac{27 w (\cD(\rho,w) w+1)}  {512
	(\cD(\rho)+1)^3} \\   \left(\frac{16 (\cD(\rho)+1)^4
			\left(\left(4-\frac{3 w (\cD(\rho,w)
				w+1)}{(\cD(\rho)+1)^2}\right)^{3/2}+\frac{9 w (\cD(\rho,w)
				w+1)}{(\cD(\rho)+1)^2}-8\right)}{27 w^2 (\cD(\rho,w) w+1)^2}-1\right).
\end{multline}

Setting $x=\rho$ and  plugging Equations~\eqref{eq:lplugme} and~\eqref{eq:hwrho} into Equation~\eqref{eq:hw}  yields an equation of the form
\begin{align}
	0 = F(w, \cD(\rho,w))
\end{align}
for an bivariate function $F(w,r)$, so that $F(w, \cD(\rho,w))$ is an explicit algebraic expression in terms of $w$, $\rho$, $\cD(\rho)$, and $\cD(\rho,w)$. The series $\cD(\rho,w)$ has radius of convergence $\ge 1$,  and it holds that $\cD(\rho,1) = \cD(\rho)$. Inserting the exact expressions for $\rho$ and $\cD(\rho)$ allows to determine that $F$ is bivariate analytic in a neighbourhood of $(1, \cD(\rho))$ and \[
	\frac{\partial F}{\partial r}(\rho, \cD(\rho)) = -0.97491341126701782727720225567\ldots.
\]
As this is non-zero, it follows by the implicit function theorem that $\cD(\rho,w)$ has radius of convergence strictly larger than $1$. 

By construction, the link weight $\iota$ may be stochastically bounded by a random variable $\iota'$ with probability generating series
\[
	\Ex{w^{\iota'}}= \frac{w + w^2\cD(\rho, w)}{1 + \cD(\rho)}
\]
As this series has radius of convergence strictly larger than $1$, it follows that $\iota$ has finite exponential moments. 
\end{proof}

\subsection{Decomposition trees}

Setting $\cM^+(x)	= \cM(x,y) / y$, we may express the specification \eqref{eq:l}--\eqref{eq:h} of 
\[
	\cD(x)	= \cH(x)  + \cL(x) + \cS(x) + \cP(x) 
\] as follows:
\begin{align*}
	\cL(x)					&= \frac{x^2}{2}(\cH(x) + \cS(x) + \cP(x) + \cI(x)) \\
	\cI(x) 				&= \frac{x^2}{4}(\cH(x)  + \cS(x) + \cP(x) + \cI(x))^2 \\
	\cS(x)					&= \sum_{k \ge 2} (\cH(x) + \cL(x) + \cP(x))^k \\
	\cP(x)					&= x^2 (\cH(x)  + \cL(x) + \cS(x) + \cP(x) ) + \frac{x^2}{2}(\cH(x)  + \cL(x) + \cS(x) + \cP(x))^2 \\
	\cH(x)					&= \cM^+(x, 1 +\cH(x)  + \cL(x) + \cS(x) + \cP(x)).
\end{align*}
Here the expression for $\cI(x)$ was obtained by combining the expressions~\eqref{eq:i} and~\eqref{eq:l}. The expression for $\cS(x)$ was obtained by unrolling~\eqref{eq:s}.
Using this specification, we may associate to any network $N$ with $n$ vertices a decomposition tree $\tau(N)$, which is a plane tree with $n$ leaves and additional labels and structures associated to each vertex. The leaves of $\tau(N)$ carry labels that correspond to the vertices of $N$. All inner vertices of $\tau(N)$ carry labels from $\{ \cL, \cI, \cS, \cP, \cH \}$. The definition is recursive and uses the decompositions illustrated in Figure~\ref{fi:total}.
\begin{itemize}
	\item Suppose that $N$ is a loop network. Then the network $N$ decomposes into two vertices $s,s'$ and another network $N'$. To be precise, there are two possible choice for $N'$, which differ only be the direction of the root-edge. We pick an arbitrary direction. We define $\tau(N)$ as the tree rooted at a vertex with label $\cL$ which has two leaf children labelled $s$ and $s'$, and which is linked by an edge to the root of $\tau(N')$.
	\item Suppose that $N$ is an isthmus network. Then it decomposes into two networks $N_1, N_2$ and two vertices $s_1', s_2'$. Again, to be precise, for $N_1$ and for $N_2$ we have actually two choices of networks, differing only by the direction of the root edge. We make an arbitrary choice of orientation. We let $\tau(N)$ be the tree rooted at a vertex with label $\cI$ which has two leaf children labelled $s$ and $s'$, and which linked by an edge to the root of $\tau(N_1)$ and by another edge to the root of $\tau(N_2)$.
	\item Suppose that $N$ is a series network. Then it decomposes into a non-series network $N'$ from $\cH + \cL + \cP$ and another network $N''$ from  $\cH + \cL + \cP + \cS$. If $N''$ is from $\cS$, we can again decompose into two parts, and so on. This process terminates after a finite number of steps. Thus, a series network $N$ consists of a concatenation of a number $k \ge 2$ of networks $N_1, \ldots, N_k$  from $\cH + \cL + \cP$. We let $\tau(N)$ denote the tree with an $\cS$-labelled root vertex that is linked via edges to the roots of $\tau(N_1), \ldots, \tau(N_k)$.
	\item Suppose that $N$ is a parallel network. Then it decomposes into two vertices $s$ and $t$ and either one or two additional networks. We define $\tau(N)$ as the tree with a root vertex labelled $\cP$ that has two leaf children labelled $s$ and $t$ and that is connected via single edges to the roots of the decomposition trees of the one or two network components.
	\item Suppose that $N$ is a polyhedral network. Then it decomposes into a $3$-connected cubic planar graph $M$ with an oriented root edge, and for each non-root edge $e$ of $M$ possibly a network $N_e$. We define the decomposition tree $\tau(N)$ as the tree rooted at a vertex with label $\cH$ that is connected to leaves carrying the vertices of $M$ as labels, and that is connected to the roots of the decomposition tree of the network components. We store additional information at the root vertex, specifically the graph $M$ together with the information which edge of $M$ corresponds to which subtree attached to the root.
\end{itemize}

\begin{remark}
	When decomposing loop networks or isthmus networks  we made arbitrary choices for the orientation of the root edge of the components. These choices only alter the order between siblings in the resulting decomposition trees. Ensuring this is the reason why we unrolled the decomposition of series networks.
\end{remark}

In the decomposition tree of a network we store at each vertex with an $\cH$-label a $3$-connected cubic planar graph with an oriented root edge. These $3$-connected graphs are the \emph{$3$-connected components} of the network. 

Note that on any $3$-connected component we have two different metrics. The  graph metric, and the \emph{subspace metric} with the distances induced from the entire network.

The singularity expansion~\eqref{eq:singhere} (implying $\cD'(\rho)<\infty$) and the  asymptotics~\eqref{eq:asymphere} allow us to apply a general result for height parameters of recursive systems of equations~\cite[Lem. 5.3]{MR2735332}, yielding:
\begin{lemma}
	\label{le:dectreeheight}
	Let $\mD_n$ denote the Boltzmann network $\mD$ conditioned on having $n$ vertices. Then for any $\epsilon>0$ there exist constants $C,c, \delta>0$ such that the height $\He(\tau(\mD_n))$ of the decomposition tree $\tau(\mD_n)$ satisfies
	\[
		\Pr{\He(\tau(\mD_n)) \ge n^{\epsilon}} \le C \exp(-c n^\delta)
	\]
	for all even integers $n \ge 4$.
\end{lemma}
\begin{proof}
	Indeed, we are in the situation of a system of equations
	\[
		\bm{y} = \bm{F}(x,\bm{y})
	\]
	with $\bm{y} = (y_1(x), \ldots, y_r(x))$ for $r = 5$, and $\bm{F}(x,y)$ an $r$-vector of bivariate functions $F_i(x, \bm{y})$ given by
	\begin{align*}
		F_1(x, \bm{y}) &= \frac{x^2}{2}(y_2 + y_3 + y_4 + y_5), \\
		F_2(x, \bm{y}) &= \frac{x^2}{4}(y_2 + y_3 + y_4 + y_5)^2, \\
		F_3(x, \bm{y}) &= \sum_{k \ge 2} (y_1 + y_4 + y_5)^k, \\
		F_4(x, \bm{y}) &= x^2(y_1 + y_3 + y_4 + y_5) + \frac{x^2}{2}(y_1 + y_3 + y_4 + y_5)^2, \\
		F_5(x, \bm{y}) &= \cM^+(x, 1 + y_1 + y_3 + y_4 + y_5).
	\end{align*}	
 	Each $F_i(x, \bm{y})$ is analytic around $(0, \bm{0})$ with non-negative coefficients in the multivariate power series expansion, with $F_i(0, \bm{y})=0$. Clearly at least one of the $F_i$ is non-affine in one of the $y_j$. The dependency graph of $\bm{F}$ (with vertex set $\{1, \ldots, r\}$ and an directed edge from $i$ to $j$ if $\frac{\partial F_j}{\partial y_i} \ne 0$) is strongly connected. The system $\bm{y} = \bm{F}(x,\bm{y})$ critical, as $\cD'(\rho)<\infty$ implies that $\bm{y}'(x)$ converges at $x=\rho$.
	
	If we start with $\bm{y}_0 := \bm{0}$, and recursively set
	\[
		\bm{y}_{h+1}(\bm{y}) := \bm{F}(x, \bm{y}_h)
	\]
	for $h \ge 0$, then the coordinates of $\bm{y}_h$ are precisely the generating function for networks (whose type corresponds to the coordinate) with a decomposition tree of height at most $h$.
	
	By the asymptotic expansion~\eqref{eq:genasymphere}, all requirements of~\cite[Lem. 5.3]{MR2735332} are met, yielding that for each class $\cF \in \{\cL, \cS, \cP, \cH\}$ such a bound holds for uniform $n$-vertex networks from $\cF$. The conditioned Boltzmann network $\mD_n$ is a uniform $n$-vertex $\cD$-network and hence a mixture of these four models. Thus, a bound of this form also holds for $\mD_n$.
\end{proof}

\subsection{Deviation bounds for the diameter}

Let us recall a standard graph-theoretic fact, proved in~\cite{zbMATH03895102}.

\begin{proposition}[{\cite[Cor. 3 of Sec. 3]{zbMATH03895102}}]
	\label{le:deledge}
	Suppose that $k \ge 1$ edges are deleted from a connected graph $G$. If the resulting graph $G'$ is still connected, then its diameter satisfies
	\[
		\Di(G') \le k + (k+1)\Di(G).
	\]
\end{proposition}

This ensures that the following deviation bound for random networks that we are going to prove also hold for the networks obtained by deleting their root edge.

\begin{lemma}
		\label{le:diamdplus}
	For any $\epsilon>0$ there exist constants $C,c, \delta>0$ such that 
	\begin{align*}
	\Pr{\Di(\mD_n) \ge n^{1/4 + \epsilon}} \le C \exp(-c n^\delta)
	\end{align*}
	for all even integers $n \ge 4$.
\end{lemma}
\begin{proof}
	Recall that an event that depends on $n$ holds with \emph{very high probability}, if the probability for its complement may be bounded by $C \exp(-c n^\delta)$ for some constants $C,c,\delta>0$ that do not depend on $n$.
	
	The network $\mD_n$ is obtained by conditioning the Boltzmann network $\mD$ on having $n$ vertices. That is, we condition on an event that by Equation~\eqref{eq:yasmp} has a polynomial probability 
	$\frac{c_{\cD}}{1 + \cD(\rho)} n^{-5/2}$. By Lemma~\ref{le:linkw}, the distance $\iota$ between the poles of $\mD$ after removing the root-edge has finite exponential moments. Hence for any $\epsilon>0$ the distance $L(\mD_n)$ between the poles of $\mD_n$ after removing the root edge  satisfies
	\begin{align}
		\label{eq:discprec}
	L(\mD_n) \le n^{\epsilon}
\end{align}
with very high probability.
	
	Note that any network has at least one $3$-connected component, since every inner vertex in the decomposition tree without another inner node as child must carry an $\cH$-label. Since $3$-connected components do not overlap, any $n$-vertex network has at most $n/4$ components. We let  $\mD_n^\circ$ be uniformly selected among all $n$-vertex $\cD$-networks with a marked $3$-connected component. Thus, for any network $D$ and any $3$-connected component $M$ of $D$ we have
	\begin{align}
		\label{eq:dumbineq}
		\Pr{\mD_n = D}  \le (n/4)\Pr{\mD^\circ_n = (D,M)}.
	\end{align}

		 Let $X_n$ denote the number of edges of the marked  component $M(\mD_n^\circ)$ of $\mD_n^\circ$. We let $(D_i(\mD_n^\circ))_{2 \le i \le X_n}$ denote the $(1+\cD)$-components inserted at the $X_n - 1$ non-root edges of $M(\mD_n^\circ)$. Thus, the $n$ vertices of $\mD_n^\circ$  are partitioned into the vertices of $M(\mD_n^\circ)$, the vertices of the components $(D_i(\mD_n^\circ))_{2 \le i \le X_n}$ inserted at its non-root vertices, and the remaining vertices of a network $D^*(\mD_n^\circ)$ attached at its root-edge. Let $\epsilon>0$.  The marked component $M(\mD_n^\circ)$ either has less than $n^{1/4 + \epsilon}$ vertices, and then necessarily diameter smaller than $n^{1/4 + \epsilon}$, or it has more than $n^{1/4 + \epsilon}$ vertices and Corollary~\ref{cor:3condeviation} ensures that its diameter is less than $n^{1/4 + \epsilon}$ with very high probability. Hence, regardless of its size, we have
		 \[
		 	\Di(M(\mD_n^\circ)) \le n^{\epsilon}
		 \] with very high probability. Since $\epsilon>0$ is arbitrary, Proposition~\ref{le:deledge} ensures that this is also the case for the diameter of the graph $M(\mD_n^\circ) -e$ obtained by deleting its root edge $e$. That is,
		 \begin{align}
		 	\label{eq:devk1}
		 \Di(M(\mD_n^\circ) - e) \le n^{\epsilon}
		 \end{align}
		 with very high probability. For each $2 \le i \le X_n$ we let $L_i$ denote the distance between the poles of the component $D_i(\mD_n^\circ)$ after removing the root-edge. If $D_i(\mD_n^\circ)$ is equal to the placeholder value indicating that we do not replace the corresponding edge of $M(\mD_n^\circ)$ by anything, we set $L_i = 0$. The diameter $\Di_{\mathrm{subspace}}(M(\mD_n^\circ))$ of $M(\mD_n^\circ)$ with respect to the subspace metric induced from $\mD_n^\circ$ admits the bound
		 \begin{align}
		 	\label{eq:devk2}
		 	\Di_{\mathrm{subspace}}(M(\mD_n^\circ)) \le (2 + \max(L_2, \ldots, L_{X_n}))\Di(M(\mD_n^\circ)-e).
		 \end{align}
		 Conditional on their size and number, the components $(D_i(\mD_n^\circ))_{2 \le i \le X_n}$ are uniform $(1+\cD)$-structures.  Hence, by Inequality~\eqref{eq:discprec}, for each $i$ the distance $L_i$ is smaller than $n^{\epsilon}$  with very high probability. As $X_n \le 3n/2$, it follows using the union bound that 
		 \begin{align}
		 	\label{eq:devk3}
		 2+\max(L_2, \ldots, L_{X_n}) \le n^{\epsilon}
		 \end{align}
		  with very high probability. Combining~\eqref{eq:devk1},~\eqref{eq:devk2}, and ~\eqref{eq:devk3} it follows that the  diameter of $M(\mD_n^\circ)$ with respect to the subspace metric induced from $\mD_n^\circ$ satisfies
		  \begin{align}
		  	\label{eq:disubgraph}
		  	\Di_{\mathrm{subspace}}(M(\mD_n^\circ)) \le n^{1/4 + 2\epsilon}
		  \end{align}
		  with very high probability.

		Let us call a $\cD$-network with a marked $3$-connected component ``bad'', if the vertex set of the marked component has diameter at least $n^{1/4 + \epsilon}$ with respect to the subspace distance. Let us call an unmarked $\cD$-network ``bad'', if it has at least  one $3$-connected component with subspace distance diameter at least  $n^{1/4 + \epsilon}$. For each unmarked  bad network $D$ we make an arbitrary choice of such a $3$-connected component and mark it, forming the marked network $D^\circ$. Using~\eqref{eq:dumbineq}, it follows that
		\begin{align}
			\label{eq:rootingargument}
			\Pr{\mD_n \text{ is ``bad''}} &\le n \sum_{D^\circ} \Pr{\mD^\circ_n = D^\circ} \\
			&\le n \Pr{\mD_n^\circ \text{ is ``bad''}}. \nonumber
		\end{align}
		By Inequality~\eqref{eq:disubgraph}, it follows that for each $\epsilon>0$ it holds that $\mD_n$ is ``good''  with very high probability in the sense that all of its $3$-connected components have subspace diameter at most $n^{1/4 + \epsilon}$.

		By Equations~\eqref{eq:asymphere} and~\eqref{eq:genasymphere} the random $\cD$-network $\mD_n$ is an $\cS$-network with probability tending to a constant $c_{\cS} / c_\cD > 0$. By~\eqref{eq:discprec}, it follows that the distance $L(\mS_n)$ between the poles in a uniform $n$-vertex $\cS$-network $\mS_n$ after removing the root-edge satisfies for each $\epsilon>0$
		\begin{align}
			\label{eq:ldiscprec}
			L(\mS_n) \le n^{\epsilon}
		\end{align}		
		with very high probability. Note that if we decompose $\mS_n$ into a sequence of (possibly more than two) non-series networks, then $L(\mS_n)$ is also an upper bound for the diameter of the collection of poles of these networks with respect to the graph distance in $\mS_n$. 

	The decomposition tree $\tau(\mD_n)$ has $n$ leaves, and any inner vertex has at least two children. Hence $\tau(\mD_n)$ has at most $n$ inner vertices. Thus, $\mD_n$ has at most $n$ ``series components'', corresponding to inner vertices of $\tau(\mD_n)$ with an $\cS$-label. Using~\eqref{eq:ldiscprec}, it follows by the same rooting argument as for~\eqref{eq:rootingargument} that for any $\epsilon>0$ with very high probability the subspace diameter of any collection of vertices corresponding to the collection of poles of the non-series components of an $\cS$-component of $\mD_n$ has diameter at most $n^{\epsilon}$.

	Now, consider the path from the root of $\tau(\mD_n)$ to one of its leaves. As illustrated in Figure~\ref{fi:total}, whenever we pass through an $\cL$, $\cI$, or $\cP$-network, we only have to cross at most $2$ edges in order to get to one of the pole vertices of the next network component. If $\mD_n$ is good as discussed for $\cH$-components and $\cS$-components, then passing through an $\cH$-network to a pole vertex of the next network component requires us to cross at most $n^{1/4 + \epsilon}$ edges, and passing through an $\cS$-network to a pole vertex of the next network component requires us to cross at most $n^{\epsilon}$ components. 	By Lemma~\ref{le:dectreeheight}, we know that 
	\begin{align}
		\label{eq:tauheightineq}
		\tau(\mD_n) \le n^{\epsilon}
	\end{align}
	with very high probability. Thus, with very high probability any vertex in $\mD_n$ may be reached from one of its poles by crossing at most $n^{1/4 + 2 \epsilon}$ edges. As $\epsilon>0$ was arbitrary, it follows that the diameter of $\mD_n$ satisfies
	\[
		\Di(\mD_n) \le n^{1/4 + \epsilon}
	\]
	with very high probability.
\end{proof}

\subsection{Components attached to the $3$-connected core}

The $3$-connected components of a network are the $3$-connected graphs encountered when forming the decomposition tree. We can form the $3$-connected components of a cubic planar graph in the same way by decomposing it, starting from an arbitrarily selected and oriented root edge. The choice does not affect the resulting $3$-connected graphs.

It was shown in~\cite{stufler2022uniform} that in a typical cubic planar graph the largest $3$-connected component has macroscopic size, with Airy-type fluctuations:

\begin{lemma}[{\cite[Thm. 1.2]{stufler2022uniform}}]
	\label{le:main2}
	Let $V_n$ denote the number of vertices in the largest $3$-connected component of the uniform random $n$-vertex cubic planar graph $\mC_n$. Let
	\[
	h(t) = \frac{1}{\pi t} \sum_{n \ge 1} (-t 3^{2/3})^n \frac{\Gamma(2n/3 +1)}{n!} \sin(-2n\pi/3), \qquad t \in \ndR
	\]
	denote the density of the map type Airy distribution. There are algebraic constants \begin{align*}
		\kappa &= 0.850853090058314333870385348879612617197477\ldots \\
		c_v &=    1.205660773457703954344217302817493214574105\ldots 
	\end{align*}
	such that for any constant $M>0$
	\begin{align*}
		\Pr{V_n = \kappa n + t n^{2/3}} = n^{-2/3} (2 c_v h(c_v t) + o(1))
	\end{align*}
	uniformly for all $t \in [-M, M]$ satisfying $\kappa n + t n^{2/3} \in 2 \ndN$. Consequently, 
	\begin{align*}
		\frac{V_n - \kappa n}{n^{2/3}} \convdis V_{3/2}
	\end{align*}
	for a $3/2$-stable random variable $V_{3/2}$  with density $c_v h(c_v t)$.
\end{lemma}
As determined in~\cite[Eq. (3.19)]{stufler2022uniform}, the exact expression for the constant $\kappa$ is given by
\begin{align}
	\kappa &= \frac{2(1 + \cD(\rho))}{2(1 + \cD(\rho)) + 3 \cD'(\rho)\rho}.
\end{align} 
Since $\kappa>1/2$ it follows that $\mC_n$ has a unique largest component with probability tending to $1$ as $n \in 2\ndN$ tends to infinity.
In this case, we may describe $\mC_n$ as the result of inserting non-isthmus network components $(\cD_i(\mC_n))_{1 \le i \le 3 V_n/2}$ at the  $3V_n/2$ edges of its  $3$-connected core $\cM(\mC_n)$. In the unlikely event that $V_n < 1/2$, we set $\cM(\mC_n)$ to some placeholder value.

Not only is $\mC_n$ likely to have a unique largest $3$-connected component, all the network components attached to it are likely to have a much smaller number of vertices:
\begin{lemma}[{\cite[Cor. 1.4]{stufler2022uniform}}]
	\label{le:compsizes}
	The numbers $|\cD_i(\mC_n)|$, $1\le i \le 3V_n/2$, of vertices in the network components attached to $\cM(\mC_n)$ satisfy
	\[
		\max_{1\le i \le 3V_n/2} |\cD_i(\mC_n)| = O_p(n^{2/3}).
	\] 
\end{lemma}

 Using the diameter bounds from Lemma~\ref{le:diamdplus}, it follows that the largest diameter of a network component has an even smaller order:

\begin{lemma}
	\label{le:compdiam}
	For any $\epsilon>0$ it holds that
	\[
		\max_{1 \le i \le 3 V_n / 2} \Di(\cD_i(\mC_n)) = o_p(n^{1/6 + \epsilon}).
	\]
\end{lemma}
\begin{proof}
	By Lemma~\ref{le:main2} we know that $V_n > n / 2$ with high probability. By Lemma~\ref{le:compsizes} we also know that 
	\[
			\max_{1\le i \le 3V_n/2} |\cD_i(\mC_n)| \le n^{2/3 + \epsilon}
	\]
	with high probability. For any even integer $k > n/2$, it holds that if  $V_n = k$ and if the maximal component size (i.e. number of vertices) is at most $n^{2/3 + \epsilon}$, then the list of numbers of vertices in the components of $\mC_n$ is a vector with values in \[
	\left\{(n_1, \ldots, n_{3k/2}) \,\Big\vert\, n_1, n_2, \ldots \in 2 \ndN_0 \cap \{0, \ldots, n^{2/3 + \epsilon}\}, \sum_{i=1}^{3k/2} n_i = n-k\right\}.
	\]
	Furthermore, conditional on $V_n = k$ and on the components  having a size vector $(n_1, \ldots, n_{3k/2})$ from this set, for each $1 \le i \le 3k/2$ the $i$th component is uniformly distributed among all $n_i$-sized $(1 + \cD)$-networks. Hence, for $n_i \le n^{1/6}$  the diameter of that component is trivially smaller than $n^{1/6 + \epsilon}$, and for $n^{1/6} \le n_i \le n^{2/3 + \epsilon}$ the conditional probability for that component to have diameter larger than $n^{1/6 +\epsilon}$ is equal to
	\begin{align}
		\label{eq:todo4535}
		\Pr{\Di(\mD_{n_i}) \ge n^{1/6 + \epsilon}} &\le \Pr{\Di(\mD_{n_i}) \ge n_i^{1/4 + \epsilon'}} 
	\end{align}
with
\[
	\epsilon' = \frac{9 \epsilon}{8+12\epsilon}.
\]
By Lemma~\ref{le:diamdplus} there are constants $C,c,\delta>0$ such that
\begin{align}
	\label{eq:const}
	\Pr{\Di(\mD_m) \ge m^{1/4 + \epsilon'}} \le C \exp(-c m^\delta)
\end{align}
for all large enough $m \in 2\ndN$.
Hence, using $n_i \ge n^{1/6}$,
	\begin{align*}
		\Pr{\Di(\mD_{n_i}) \ge n_i^{1/4 + \epsilon'}} &\le C \exp(-c n_i^\delta) \\
		& \le C \exp(-c n^{\delta/6}).
	\end{align*}
	Thus, 
	\begin{multline}
		\Prb{\max_{1 \le i\le 3k/2} \Di(\cD_i(\mD_n)) \ge n^{1/6 + \epsilon} \mid V_n = k, \max_{1\le i \le 3V_n/2} |\cD_i(\mC_n)| \le n^{2/3 + \epsilon}} \\\le (3k/2)C  \exp(-c n^{\delta/6}).
	\end{multline}
	As $V_n > n/2$ and $\max_{1\le i \le 3V_n/2} |\cD_i(\mC_n)| \le n^{2/3 + \epsilon}$ both hold with high probability, it follows that
	\[
			\max_{1 \le i \le 3 V_n / 2} \Di(\cD_i(\mC_n)) < n^{1/6 + \epsilon}
	\]
	holds with high probability. As $\epsilon>0$ was arbitrary, the proof is complete.
	%	Let $(\mD(i))_{i \ge 1}$ denote independent copies of the Boltzmann network $\mD$, and let $|\mD(i)|$ denote the number of vertices of $\mD(i)$.
	%\begin{align}
	%	\label{eq:zy}
	%	\left( (\cD_i(\mC_n))_{1\le i \le 3V_n/2} \mid V_n = k \right) \eqdist \left( (\mD(i))_{1\le i \le 3V_n/2} \,\Big\vert\, \sum_{i=1}^{3k/2} |\mD(i)| = n-k \right).
	%\end{align}
\end{proof}

We emphasize that by Proposition~\ref{le:deledge}, the statement of Lemma~\ref{le:compdiam} also holds for the diameters of the components after removing their root edges.

\subsection{Proof of Theorem~\ref{te:main}}

The  local limit theorem from Lemma~\ref{le:main2} ensures that the core $\cM(\mC_n)$ behaves like a mixture of the random $3$-connected cubic planar graphs $(\mM_n)_{n \ge 1}$ with random size  $\kappa n + O_p(n^{2/3})$. Using Theorem~\ref{te:cubic3con}, we   deduce a scaling limit for $M(\mC_n)$ with respect to the graph distance $d_{\cM(\mC_n)}$ and the uniform measure $\mu_{\cM(\mC_n)}$ on its vertices:
\begin{corollary}
	\label{co:coreconvprel}
	There exists a constant $c^\dagger>0$ such that
	\begin{align*}
		\left(\cM(\mC_n), \frac{3^{1/4}}{2 c^\dagger} (\kappa n)^{-1/4} d_{\cM(\mC_n)}, {\mu}_{\cM(\mC_n)}\right) \convdis (\mathbf{M}, d_{\mathbf{M}}, \mu_{\mathbf{M}})
	\end{align*}
	in the Gromov--Hausdorff--Prokhorov sense as $n \in 2\ndN$ tends to infinity.
\end{corollary}
\begin{proof}
	Since $V_n \convp \infty$, Theorem~\ref{te:cubic3con} immediately yields
	\begin{align}
		\label{eq:prelcorconv}
	\left(\cM(\mC_n), \frac{3^{1/4}}{2 c^\dagger} V_n^{-1/4} d_{\cM(\mC_n)}, {\mu}_{\cM(\mC_n)}\right) \convdis (\mathbf{M}, d_{\mathbf{M}}, \mu_{\mathbf{M}}).
	\end{align}
	By Proposition~\ref{pro:ghpcoupling}, it follows that the Gromov--Hausdorff--Prokhorov distance
	\[
		d_{\mathrm{GHP}}\left(\left(\cM(\mC_n), \frac{3^{1/4}}{2 c^\dagger (\kappa n)^{1/4}}  d_{\cM(\mC_n)}, {\mu}_{\cM(\mC_n)}\right), \left(\cM(\mC_n), \frac{3^{1/4}}{2 c^\dagger V_n^{1/4}}  d_{\cM(\mC_n)}, {\mu}_{\cM(\mC_n)}\right) \right)  
	\]
	is bounded by
	\begin{align}
		\label{eq:upperboundprel}
		\Di(\cM(\mC_n))\left(\frac{3^{1/4}}{2 c^\dagger (\kappa n)^{1/4}} - \frac{3^{1/4}}{2 c^\dagger V_n^{1/4}} \right) = \frac{3^{1/4}}{2 c^\dagger}  \frac{\Di_{\cM(\mC_n)}}{V_n^{1/4}} \left(1 - \left( \frac{ V_n}{\kappa n} \right)^{1/4} \right).
	\end{align}
	The diameter is a continuous functional with respect to the GHP-distance. Hence, by~\eqref{eq:prelcorconv},
	\[
		\frac{3^{1/4}}{2 c^\dagger}  \frac{\Di_{\cM(\mC_n)}}{V_n^{1/4}} \convp \Di(\mathbf{M}).
	\]
	Lemma~\ref{le:main2} ensures that the other factor satisfies
	\[
		1 - \left( \frac{ V_n}{\kappa n} \right)^{1/4}  \convp 0.
	\]
	By Slutsky's theorem, it follows that the upper bound in~\eqref{eq:upperboundprel} converges in probability to zero as $n \in 2 \ndN$ tends to infinity. Hence, by~\eqref{eq:prelcorconv}, 
		\begin{align*}
		\left(\cM(\mC_n), \frac{3^{1/4}}{2 c^\dagger} (\kappa n)^{-1/4} d_{\cM(\mC_n)}, {\mu}_{\cM(\mC_n)}\right) \convdis (\mathbf{M}, d_{\mathbf{M}}, \mu_{\mathbf{M}})
	\end{align*}
\end{proof}

The ordering of the edges in the largest $3$-connected component $\cM(\mC_n)$ was done in an arbitrary canonical way, so that we may refer to the $i$th component $\cD_i(\mC_n)$ for $1 \le i \le 3 V_n/2$. This is always possible since the vertices of $\cM(\mC_n)$ carry labels. Since the components are exchangeable, the choice of ordering is irrelevant as long as it only depends on the core itself.

 It will be convenient to assume from now on that this ordering was done  according to a breadth-first-search.  That is, for every finite $3$-connected cubic planar graph $M$ we fix a  root vertex $o_M$ and a breadth-first-search ordering of its edges starting from $o_M$,  so that the distance of the edges from the point $o$ increases monotonically. We let  $o = o_{\cM(\mC_n)} \in \cM(\mC_n)$ denote the fixed root of $\cM(\mC_n)$.

The components $(\cD_i(\mC_n))_{1 \le i \le 3 V_n/2}$ depend on each other and on $V_n$, since the sum of $V_n$ and the number of vertices in the components needs to equal exactly $n$. However, it was shown in~\cite{stufler2022uniform} that $\mC_n$ satisfies a contiguity relation to a model where a large part of these components are resampled independently according to the Boltzmann  network model $\mD$ described in Definition~\ref{de:defd}.

\begin{lemma}[{\cite[Thm. 1.3]{stufler2022uniform}}]
	\label{le:main3}
	Let $(\mD(i))_{i \ge 1}$ denote independent copies of the Boltzmann  network $\mD$.
	For any $\epsilon >0$ and  $0< \delta < 3 \kappa / 2$ there exist constants $0<c<C$ and $N >0$ and sets $(\cE_n)_{n \ge N}$ such that for all $n \in 2 \ndN$ with $n \ge N$ 
	\begin{align*}
		\Prb{ (\cM(\mC_n), (\cD_i(\mC_n))_{\delta n \le i \le 3V_n/2 }  \notin \cE_n  } < \epsilon
	\end{align*}
	and
	\begin{align*}
		\Prb{ (\cM(\mC_n),  (\mD(i))_{\delta n \le i \le 3V_n/2 }) \notin \cE_n  } < \epsilon
	\end{align*}
	and for all elements $E \in \cE_n $
	\begin{align*}
		c < \frac{\Prb{ (\cM(\mC_n), (\cD_i(\mC_n))_{\delta n \le i \le 3V_n/2  })  =E  }} {	\Prb{ (\cM(\mC_n),  (\mD(i))_{\delta n \le i \le 3V_n/2  }) =E  }}  < C.
	\end{align*}
\end{lemma}
Note that since the components $(\cD_i(\mC_n))_{1 \le i \le 3 V_n/2}$ are exchangeable, it is irrelevant if we consider the subfamily for $\delta n \le i \le 3V_n/2$ or the subfamily for $1 \le i \le 3V_n/2 - \delta n $.

Our first observation states that taking $\delta>0$ small means that the area $\Omega_\delta$ consisting of the first $\lfloor \delta n \rfloor $ components of $\mC_n$ together with the corresponding subset in $\cM(\mC_n)$ is likely to have small diameter.

\begin{lemma}
	\label{le:suffscmall}
	Let $\Omega_\delta \subset \cM(\mC_n)$ denote the subset of vertices consisting of the endpoints of the first $\lfloor \delta n \rfloor$ edges. For all $\epsilon>0$ we may select $\delta>0$ small enough, such that 
	\[
		\Pr{ \Di(\Omega_\delta) \le \epsilon n^{1/4} } \ge 1 - \epsilon
	\]
	for all sufficiently large $n \in 2 \ndN$.
\end{lemma}
\begin{proof}
	Recall the notation for closed balls from~\eqref{eq:mbal}.
	Since we ordered the edges in the $3$-connected core $\cM(\mC_n)$ in a breadth-first-search manner starting from a vertex $o$, it follows that there exists an integer $R \ge 0$ with
	\begin{align}
		\label{eq:inclus}
		C_R^{\cM(\mC_n)}(o) \subset \Omega_\delta \subset C_{R+1}^{\cM(\mC_n)}(o).
	\end{align}
	It follows that $C_R^{\cM(\mC_n)}(o)$ has at most $\delta n$ edges, and consequently at most $2 \delta n$ vertices. Hence
	\begin{align}
		\label{eq:aaprel}
		\mu_{\cM(\mC_n)}(C_R^{\cM(\mC_n)}(o)) \le 2 \delta.
	\end{align}
	As recalled in Section~\ref{sec:brownianmap}, the Brownian map almost surely has full support. By Corollary~\ref{co:fullrandom} and Corollary~\ref{co:coreconvprel}, it follows that for all $\epsilon_1, \epsilon_2 >0$ there exists a $\delta_1>0$ such that for all sufficiently large $n$
	\begin{align}
		\label{eq:bbprel}
		\Prb{ \inf_{v \in \cM(\mC_n)} \mu_{\cM(\mC_n)} (C_{\epsilon_1}^{\cM(\mC_n)}(v)) \ge \delta_1} > 1 - \epsilon_2. 
	\end{align}
	Hence, by~\eqref{eq:aaprel} and~\eqref{eq:bbprel}, for all $\delta < \delta_1/2$, we have
	\[
		\Prb{R \frac{3^{1/4}}{2 c^\dagger} (\kappa n)^{-1/4} < \epsilon_1 } > 1 - \epsilon_2.
	\]
	By~\eqref{eq:inclus} we have
	\[
		\Di(\Omega_\delta)  \le R+1,
	\]
	Thus, taking $\epsilon_2= \epsilon$, $\epsilon_1= \frac{3^{1/4}}{4 c^\dagger} \kappa ^{-1/4} \epsilon $  we arrive at
	\[
				\Prb{\Di(\Omega_\delta) < \epsilon} > 1 - \epsilon
	\]
	for all sufficiently large $n$ and all $0 < \delta < \delta_1/2$.
\end{proof}

Recall the definition of the link weight $\iota \ge 1$ from Definition~\ref{de:linkweight}. By Lemma~\ref{le:fpp3con} there exists a corresponding constant $c_{\mathrm{fpp}}>0$ such that
\begin{align}
	\label{eq:thefpponcorefor}
	n^{-1/4} \sup_{u,v \in \mM_n} \left | c_{\mathrm{fpp}} d_{\mM_n}(u,v) -  d_{\mathrm{fpp}}(u,v)\right | \convp 0
\end{align}
as $n \in 2\ndN$ tends to infinity.

\begin{lemma}
	\label{le:hauptlemma}
	We have
	\[
		\sup_{x,y \in \cM(\mC_n)} | d_{\mC_n}(x,y) - c_{\mathrm{fpp}} d_{\cM(\mC_n)}(x,y) | = o_p(n^{1/4}).
	\]
\end{lemma}
\begin{proof} 
	The idea of the proof is that a small linear fraction of the approximately $3 \kappa n/ 2$ components inserted at the $3$-connected core causes a small distortion of distances, and that concentration inequalities that hold with high probability for distances expressed in terms of only the core and the remaining components may be transferred using Lemma~\ref{le:main3}. One of the difficulties is the subtle interplay between masses and radii of balls. In particular, Lemma~\ref{le:suffscmall} does not state some property to hold with high probability,  hence a very careful reasoning is required to make the contiguity argument work.
	
	We proceed in two parts. In the first part, we control the first $\kappa n$ components attached to the core $\cM(\mC_n)$ and show that a small linear fraction $\delta n$ of edges of $\cM(\mC_n)$ is likely to have a small diameter in both $\mC_n$ and a model where the components inserted at $\cM(\mC_n)$ are i.i.d. copies of the Boltzmann network $\mD$.  In the second part, we contract the $\delta n$ edges to a single point and control the remaining components with index $\delta n \le i \le 3 V_n/2$ in order to bound the distortion between  the subspace distance on $\cM(\mC_n)$ induced from $\mC_n$ and the subspace distance on $\cM(\mC_n)$ induced from the graph obtained by inserting i.i.d. copies of $\mD$ at all edges of $\cM(\mC_n)$.

	Let us start with the first part. By Lemma~\ref{le:main3} and exchangeability of components attached to the $3$-connected core, for each $\epsilon_1>0$ there exist constants $c_1,C_1 >0$ (that do not depend on $n$) such that for all sufficiently large $n \in 2\ndN$ there exists a set  $\cF_n$ with
	\begin{align}
		\label{eq:bach1}
		\Prb{ (\cM(\mC_n), (\cD_i(\mC_n))_{1 \le i \le \min(\kappa n ,3V_n/2) }  \notin \cF_n  } < \epsilon_1
	\end{align}
	and
	\begin{align}
		\label{eq:bach2}
		\Prb{ (\cM(\mC_n),  (\mD(i))_{1\le i \le \min(\kappa n , 3V_n/2) }) \notin \cF_n  } < \epsilon_1
	\end{align}
	and uniformly for all $F \in \cF_n $
	\begin{align}
		\label{eq:indeqaul}
		c_1 < \frac{\Prb{ (\cM(\mC_n), (\cD_i(\mC_n))_{1 \le i \le \min(\kappa n /2,V_n)  })  = F  }} {	\Prb{ (\cM(\mC_n),  (\mD(i))_{1 \le i \le \min(\kappa n , 3V_n/2)  }) = F  }}  < C_1.
	\end{align}
	Recall the notation for open balls in~\eqref{eq:openball}. By Lemma~\ref{le:conconlambda} and Corollary~\ref{co:coreconvprel},  for each $\epsilon_2>0$ there is a $\delta_2>0$  with
	\begin{align}
		\Prb{ \sup_{u \in \cM(\mC_n)} \mu_{\cM(\mC_n)}\left( B_{\delta_2 n^{1/4}}^{\cM(\mC_n)}(u)\right) > 1/6 } < \epsilon_2.
	\end{align}
	The notational choice of index $2$ is to emphasize that $\delta_2$ depends on $\epsilon_2$.  We assume $\epsilon_2$ to be small enough so that
	\begin{align}
		\label{eq:handelimmed0}
		4 \epsilon_2 C_2 < \epsilon_1.
	\end{align}
	Let $\epsilon>0$  be small enough so that
	\begin{align}
		\label{eq:handelimmed}
		(3c_{\mathrm{fpp}}^{-1}+4)\epsilon < \delta_2 \qquad \text{and} \qquad 4 \epsilon C_2 <  \epsilon_1.
	\end{align}
	Recall that we fixed a vertex $o$ and a breadth-first-search ordering of the edges of $\cM(\mC_n)$ so that the $d_{\cM(\mC_n)}$-distance of an edge from $o$ increases monotonically with its index. By~\eqref{eq:handelimmed}, 
	\begin{align}
		\label{eq:lucifer}
		\Prb{ \mu_{\cM(\mC_n)}( B_{(3c_{\mathrm{fpp}}^{-1}+4)\epsilon n^{1/4}}^{\cM(\mC_n)}(o)) > 1/6 } < \epsilon_2.
	\end{align}
	If a subset of the cubic $3$-connected planar graph $\cM(\mC_n)$ has at most $V_n/6$ vertices, then it is incident to at most $V_n/2$ edges. 	
	Since the minimum of the distances of the endpoints of an edge from $o$ increases monotonically with its index, if follows by Lemma~\ref{le:main2} and Inequality~\eqref{eq:lucifer} that
	\begin{align}
		\label{eq:lucifer2}
		\Prb{  B_{(3c_{\mathrm{fpp}}^{-1}+4)\epsilon n^{1/4}}^{\cM(\mC_n)}(o) \text{ incident to any edge with index $ i > \kappa n$} } < 2\epsilon_2
	\end{align}
	for all sufficiently large $n$.  By Lemma~\ref{le:suffscmall} we may choose $\delta>0$ small enough so that for all large enough $n \in 2\ndN$
	\begin{align}
		\label{eq:handel}
		\Pr{ \Di(\Omega_\delta) < \epsilon n^{1/4} } \ge 1 - \epsilon.
	\end{align}
	Since $o \in \Omega_\delta$, this means \[
	\Prb{\Omega_\delta \subset B_{\epsilon n^{1/4}}^{\cM(\mC_n)}(o)} \ge 1 - \epsilon
	\]
	for large enough $n$.

	 We let $d_{\mathrm{fpp}}$ denote the subspace distance on $\cM(\mC_n)$ induced by the cubic planar graph described by the components $(\cM(\mC_n),  (\mD(i))_{1 \le i \le 3V_n/2})$. Since $\iota$ is precisely the length of a link after inserting the Boltzmann network $\mD$, it follows that $d_{\mathrm{fpp}}$ is distributed like a first passage percolation metric with link weight $\iota$.

	Using~\eqref{eq:thefpponcorefor} and Lemma~\ref{le:main2}, it follows by identical arguments as for Corollary~\ref{co:coreconvprel} that 
	\begin{align}
		\label{eq:devil1}
		\Prb{ \sup_{u,v \in \cM(\mC_n)} \left | c_{\mathrm{fpp}} d_{\cM(\mC_n)}(u,v) -  d_{\mathrm{fpp}}(u,v)\right | > \epsilon  n^{1/4} } \to 0
	\end{align}
	as $n \in 2 \ndN$ tends to infinity. 
	
	Letting $B_{(\cdot)}^{\mathrm{fpp}}(\cdot)$ denote open balls in $\cM(\mC_n)$ with respect to $d_{\mathrm{fpp}}$, it follows by~\eqref{eq:devil1} that
	\begin{align}
		\label{eq:prepboundA}
		\Prb{\Omega_\delta \subset B_{ (c_{\mathrm{fpp}}+1)\epsilon n^{1/4}}^{\mathrm{fpp}}(o)} \ge 1- 2 \epsilon
	\end{align}
	for large enough $n$,
	and
	\begin{align}
		\label{eq:cyphill}
		\Prb{ B_{ 3(c_{\mathrm{fpp}}+1)\epsilon n^{1/4}}^{\mathrm{fpp}}(o) \subset B_{ (3c_{\mathrm{fpp}}^{-1}+4)\epsilon n^{1/4}}^{\cM(\mC_n)}(o)} = 1 + o(1).
	\end{align}
	Combining~\eqref{eq:cyphill} with~\eqref{eq:lucifer2} we obtain
	\begin{align}
				\label{eq:prepboundB}
		\Prb{  B_{ 3(c_{\mathrm{fpp}}+1)\epsilon n^{1/4}}^{\mathrm{fpp}}(o) \text{ incident to any edge of $\cM(\mC_n)$ with index  $i>\kappa n$ } } < 3\epsilon_2
	\end{align}
	for large enough $n$. By~\eqref{eq:prepboundA} and~\eqref{eq:prepboundB} with probability at least $1 - 2 \epsilon - 3 \epsilon_2$ any point in $\Omega_\delta$ can be reached from the point $o$ by path in the graph corresponding to $(\cM(\mC_n),  (\mD(i))_{1\le i \le  3V_n/2})$ that only passes through  components $\mD(i)$ with index~$i \le \kappa n$.

	Using~\eqref{eq:bach2} and~\eqref{eq:indeqaul}, it follows that 
	\begin{align}
		\Pr{ d_{\mC_n}(o, \Omega_\delta) > \epsilon n^{1/4}} \le \epsilon_1  + C_1(2 \epsilon + 3 \epsilon_2)
	\end{align}
	for sufficiently large $n \in 2\ndN$. Hence, by Inequalities~\eqref{eq:handelimmed0} and~\eqref{eq:handelimmed}, 
	\begin{align}
		\label{eq:krypt}
		\Pr{ d_{\mC_n}(o, \Omega_\delta) > \epsilon n^{1/4}} \le 2 \epsilon_1.
	\end{align}
	This concludes the first part of the proof.
	
	We proceed with the second part. By Lemma~\ref{le:main3} there exist constants $c,C>0$ (that do not depend on $n$) such that for all sufficiently large $n \in 2\ndN$ there exists a set  $\cE_n$ with
	\begin{align}
		\label{eq:handel1}
		\Prb{ (\cM(\mC_n), (\cD_i(\mC_n))_{\delta n \le i \le 3V_n/2 }  \notin \cE_n  } < \epsilon
	\end{align}
	and
	\begin{align}
		\label{eq:handel2}
		\Prb{ (\cM(\mC_n),  (\mD(i))_{\delta n \le i \le 3V_n/2 }) \notin \cE_n  } < \epsilon
	\end{align}
	and for all elements $E \in \cE_n $
	\begin{align}
		\label{eq:ineqcontigu}
		c < \frac{\Prb{ (\cM(\mC_n), (\cD_i(\mC_n))_{\delta n \le i \le 3V_n/2  })  =E  }} {	\Prb{ (\cM(\mC_n),  (\mD(i))_{\delta n \le i \le 3V_n/2  }) =E  }}  < C.
	\end{align}
	We let $\tilde{\cE}_n \subset \cE_n$ denote the subset of all sequences $E= (M, (D_i)_{\delta n \le i \le 3V/2})$ such that the endpoints of the first $\lfloor \delta n \rfloor$ edges of $M$ in the breadth-first-search order starting from the vertex $o_M$ form a subset with diameter less than $\epsilon n^{1/4}$. By~\eqref{eq:handel},~\eqref{eq:handel1}, and ~\eqref{eq:handel2} it follows that
	\begin{align}
	\label{eq:keynote1}
	\Prb{ (\cM(\mC_n), (\cD_i(\mC_n))_{\delta n \le i \le 3V_n/2 }  \notin \tilde{\cE}_n  } < 2\epsilon
\end{align}
and
\begin{align}
	\label{eq:keynote2}
	\Prb{ (\cM(\mC_n),  (\mD(i))_{\delta n \le i \le 3V_n/2 }) \notin \tilde{\cE}_n  } < 2\epsilon.
\end{align}

	For any $\delta>0$  we let $\cM^\delta(\mC_n)$ denote the graph obtained by contracting the subset $\Omega_\delta \subset \cM(\mC_n)$ to a single point. Any edge with one end in $\Omega_\delta$ and the other outside of $\Omega_\delta$  becomes an edge incident to this contracted point. We let  $d^\delta_{\cM(\mC_n)}$ denote the graph distance on $\cM^\delta(\mC_n)$. We let $d^\delta_{\mC_n}$ denote the distance on  $\cM^\delta(\mC_n)$ obtained by taking $\mC_n$,  $\Omega_\delta$ contracting $\Omega^\delta$ to a single point, and taking the subspace distance on~$\cM^\delta(\mC_n)$.

	 We also define the distance $d_{\mathrm{fpp}}^\delta$ on $\cM^\delta(\mC_n)$ obtained by taking the graph described by $(\cM(\mC_n),  (\mD(i))_{1 \le i \le 3V_n/2})$, contracting $\Omega_\delta$ to a point, and forming the subspace distance on the subset $\cM^\delta(\mC_n)$.

In order to shorten notation, we define the events
  \[
 \cA= \left\{ (\cM(\mC_n), (\cD_i(\mC_n))_{\delta n \le i \le 3V_n/2 })  \in \tilde{\cE}_n \right\}
 \]
 and
 \[
\cB = \left\{ (\cM(\mC_n),  (\mD(i))_{\delta n \le i \le 3V_n/2 }) \in \tilde{\cE}_n \right \}. 
\]
By definition of $\tilde{\cE}_n$, each of these events entails that \[
\Di(\Omega_\delta) < \epsilon n^{1/4}.\]

Recall that by~\eqref{eq:devil1}
\[
	\Prb{ \sup_{u,v \in \cM(\mC_n)} \left | c_{\mathrm{fpp}} d_{\cM(\mC_n)}(u,v) -  d_{\mathrm{fpp}}(u,v)\right | > \epsilon  n^{1/4} } \to 0
\]
as $n \in 2\ndN$ becomes large. If the inequality $\Di(\Omega_\delta) \le \epsilon n^{1/4}$ holds simultaneously with \[\sup_{u,v \in \cM(\mC_n)} \left | c_{\mathrm{fpp}} d_{\cM(\mC_n)}(u,v) -  d_{\mathrm{fpp}}(u,v)\right | \le \epsilon  n^{1/4},\] then 
\[
	\sup_{u,v \in \Omega_{\delta}} d_{\mathrm{fpp}}(u,v) \le (\epsilon + \epsilon c_{\mathrm{fpp}}) n^{1/4}.
\] 
 In this case, for all vertices $u,v \in \cM(\mC_n) - \Omega_\delta$ (that is, all vertices of $\cM(\mC_n)$ except those in $\Omega_\delta$) 
\[
	|d_{\mathrm{fpp}}(u,v) - d_{\mathrm{fpp}}^\delta(u,v)| \le (\epsilon + \epsilon c_{\mathrm{fpp}}) n^{1/4}.
\]
(To see this, note that if a geodesic from $u$ to $v$ intersects with $\Omega_\delta$, then it intersects at at a first point $a$ and a last point $b$, and $d_{\mathrm{fpp}}(a,b) \le (\epsilon + \epsilon c_{\mathrm{fpp}}) n^{1/4}$. Contracting $\Omega_\delta$ hence shortens the distance between $u$ and $v$ by at most $(\epsilon + \epsilon c_{\mathrm{fpp}})n^{1/4}$.)
Using~\eqref{eq:devil1} and the triangle inequality, it follows that
\begin{align}
	\label{eq:devil2}
	 	 \Prb{\cB,   \sup_{u,v \in \cM(\mC_n) - \Omega_\delta } \left | c_{\mathrm{fpp}} d_{\cM(\mC_n)}(u,v) -  d_{\mathrm{fpp}}^\delta(u,v)\right | >( 2 \epsilon + \epsilon c_{\mathrm{fpp}}) n^{1/4} } \to 0
\end{align}
as $n \in 2 \ndN$ tends to infinity. 	 Note that $d^\delta_{\mC_n}$ is fully determined by the components $(\cM(\mC_n), (\cD_i(\mC_n))_{\delta n \le i \le 3V_n/2 })$, and  $d_{\mathrm{fpp}}^\delta$ is fully determined by the components $(\cM(\mC_n),  (\mD(i))_{\delta n \le i \le 3V_n/2 })$. 
Using Inequality~\eqref{eq:ineqcontigu} and~\eqref{eq:devil2} it follows that 
\begin{align}
	\Prb{\cA,   \sup_{u,v \in \cM(\mC_n) - \Omega_\delta } \left | c_{\mathrm{fpp}} d_{\cM(\mC_n)}(u,v) -  d_{\mC_n}^\delta(u,v)\right | >( 2 \epsilon + \epsilon c_{\mathrm{fpp}}) n^{1/4} } \to 0.
\end{align}
By~\eqref{eq:keynote1}, it follows that 
\begin{align}
	\label{eq:apo1}
	\Prb{\sup_{u,v \in \cM(\mC_n) - \Omega_\delta } \left | c_{\mathrm{fpp}} d_{\cM(\mC_n)}(u,v) -  d_{\mC_n}^\delta(u,v)\right | >( 2 \epsilon + \epsilon c_{\mathrm{fpp}}) n^{1/4} } \le 2\epsilon
\end{align}
for all sufficiently large $n$. If $d_{\mC_n}(o, \Omega_\delta) < \epsilon n^{1/4}$, then contracting $\Omega_\delta$ changes distances by at most $2\epsilon$. Hence, by Inequality~\eqref{eq:krypt} we know that
\begin{align}
	\label{eq:apo2}
	\Prb{ \sup_{u,v \in \cM(\mC_n) - \Omega_\delta } \left |  d_{\cM(\mC_n)}(u,v) -  d_{\mC_n}^\delta(u,v)\right | > 2 \epsilon n^{1/4}} \le 2 \epsilon_1.
\end{align}
Combining~\eqref{eq:apo1} and \eqref{eq:apo2} using the triangle inequality it follows that
\begin{align}
	\label{eq:apo3}
	\Prb{\sup_{u,v \in \cM(\mC_n) - \Omega_\delta } \left | c_{\mathrm{fpp}} d_{\cM(\mC_n)}(u,v) -  d_{\mC_n}(u,v)\right | >( 4\epsilon + \epsilon c_{\mathrm{fpp}}) n^{1/4} } \le 2\epsilon + 2\epsilon_1.
\end{align}
Note that $d_{\mC_n}(o, \Omega_\delta) < \epsilon n^{1/4}$ entails $d_{\cM(\mC_n)}(o, \Omega_\delta) < \epsilon n^{1/4}$, because we cannot decrease  distances by inserting components. Hence, using again Inequality~\eqref{eq:krypt}, we get
\begin{align}
	\label{eq:apo4}
	\Prb{\sup_{u,v \in \cM(\mC_n) } \left | c_{\mathrm{fpp}} d_{\cM(\mC_n)}(u,v) -  d_{\mC_n}(u,v)\right | >( 6\epsilon + 3\epsilon c_{\mathrm{fpp}}) n^{1/4} } \le 4\epsilon + 2\epsilon_1.
\end{align}
This completes the proof.
\end{proof}

Lemma~\ref{le:hauptlemma} and the scaling limit for $\cM(\mC_n)$ in Corollary~\ref{co:coreconvprel} immediately yield:

\begin{corollary}
	\label{co:comeonalmost}
		As $n \in 2\ndN$ tends to infinity,
		\[
		\left(\cM(\mC_n), \frac{3^{1/4}}{2 c^\dagger c_{\mathrm{fpp}} } (\kappa n)^{-1/4} d_{\mC_n}, \mu_{\cM(\mC_n)}\right) \convdis (\mathbf{M}, d_{\mathbf{M}}, \mu_{\mathbf{M}})
		\]
		in the Gromov--Hausdorff--Prokhorov sense.
\end{corollary}

By Lemma~\ref{le:compdiam} we know that the Hausdorff distance between the random connected cubic planar graph $\mC_n$ and its $3$-connected core $\cM(\mC_n)$ tends in probability to zero after rescaling distances by $n^{-1/4}$. Using Corollary~\ref{co:comeonalmost} it follows that
\begin{align*}
	\left(\mC_n, \frac{3^{1/4}}{2 c^\dagger c_{\mathrm{fpp}} } (\kappa n)^{-1/4} d_{\mC_n}\right) \convdis (\mathbf{M}, d_{\mathbf{M}})
\end{align*}
in the Gromov--Hausdorff sense. However, in order to show convergence in the Gromov--Hausdorff--Prokhorov sense we need to bound the Prokhorov distance between the uniform measures on $\mC_n$ and $\cM(\mC_n)$ after rescaling distances. 

Since the network components $(\cD_i(\mC_n))_{\delta n \le i \le 3V_n/2 }$ are exchangeable conditional on $V_n$, we may argue analogous as in the proof of a similar result~\cite[Cor. 6.2]{MR3729639} for random quadrangulations. In order to provide details, we recall a general lemma given in~\cite{MR3729639}:

\begin{proposition}[{\cite[Lem. 5.3]{MR3729639}}]
	\label{eq:prono}
	Let $G$ be a connected planar graph with an ordering $e_1, \ldots, e_m$ of its edges.  Let $\bm{n} = (n_1, \ldots, n_m)$ be an exchangeable random vector of non-negative real numbers.  For each $1 \le i \le m$ let $w_i$ denote a uniformly and independently selected endpoint of the edge $e_i$. Suppose that $\norm{\bm{n}}_1>0$ and define the random probability measure $\nu_G^{\bm{n}}$ on the vertex set of $G$ with 
	\[
		\nu_G^{\bm{n}}(U) = \norm{\bm{n}}^{-1}_1 \sum_{i \,:\, w_i \in U } n_i
	\]
	for any subset $U$ of the vertex set of $G$. Here $\norm{ \cdot }_p$ denotes the $p$-norm for any $p >0$.
	 Then for any $t>0$ 
	\begin{multline*}
	\Prb{ \left| (2m)^{-1} \sum_{u \in U} d_G(u) - \nu_G^{\bm{n}}(U) \right| > \frac{2t}{\norm{\bm{n}}_1} + \frac{1}{m}  \,\, \Bigg \vert \,\,  \norm{\bm{n}}_2	 } 
	\le 4 \exp \left(- \frac{2t^2}{\norm{\bm{n}}_2^2} \right).
	\end{multline*}
\end{proposition}

We are specifically interested in the case where $G = \cM(\mC_n)$ and where the exchangeable random vector represents the number of vertices in the attached components. From here on, in order to simplify notation, we set
\begin{align}
	B_{\epsilon}(u) = \left\{ v \in \cM(\mC_n) \,\, \Bigg \vert \,\, \frac{3^{1/4}}{2 c^\dagger c_{\mathrm{fpp}} } (\kappa n)^{-1/4} d_{\mC_n}(u,v) < \epsilon \right \}
\end{align}
for all $\epsilon>0$ and all vertices $u \in \cM(\mC_n)$.

\begin{corollary}
	\label{co:noffff}
	For each $1 \le i \le 3 V_n / 2$, let $n_i = |\cD_i(\mC_n)|$ denote the number of vertices in the $i$th network component of $\mC_n$, and set $\bm{n} = (n_i)_{1 \le i \le 3 V_n / 2}$. Let $u$ and $v$ denote uniform random independent vertices of $\cM(\mC_n)$. Then for each $\epsilon>0$
	\begin{align*}
		\left| \nu_{\cM(\mC_n)}^{\bm{n}}\left(B_{\epsilon}(u)\right) - \mu_{\cM(\mC_n)}\left(B_{\epsilon}(u)\right)  \right| \convp 0, \\
		\left| \nu_{\cM(\mC_n)}^{\bm{n}}\left(B_{\epsilon}(u) \cap B_{\epsilon}(v) \right) - \mu_{\cM(\mC_n)}\left(B_{\epsilon}(u) \cap B_{\epsilon}(v) \right)   \right| \convp 0,
	\end{align*}
as $n \in 2\ndN$ tends to infinity.
\end{corollary}
\begin{proof}
	Using Lemma~\ref{le:main2}, it follows that
	\begin{align}
		\label{eq:noni1}
		\norm{\bm{n}}_1 = n - V_n \convp (1- \kappa) n.
	\end{align}
	By Lemma~\ref{le:compsizes} we know
	\[
		\max_{1\le i \le 3V_n/2} |\cD_i(\mC_n)| = O_p(n^{2/3})
	\]
	and consequently
	\begin{align}
		\label{eq:noni2}
		\norm{\bm{n}}_2 = O_p(n^{5/6}).
	\end{align}
	Note that for cubic graphs, the uniform measure is identical to the stationary distribution, which describes a random vertex sampled with probability proportional to its degree.  Let $U$ either denote the subset $B_{\epsilon}(u)$ or the subset $B_{\epsilon}(u) \cap B_{\epsilon}(v)$. Using Proposition~\ref{eq:prono} for $t = n^{1/12}\norm{\bm{n}}_2$ it follows that 
	\begin{align*}
		\Prb{ \left| \mu_{\cM(\mC_n)}(U) - \nu_{\cM(\mC_n)}^{\bm{n}}(U) \right| > 2 n^{1/12}\frac{\norm{\bm{n}}_2}{\norm{\bm{n}}_1} + \frac{2}{3V_n}  } 
		\le 4 \exp \left(- 2 n^{1/6} \right).
	\end{align*}
	By Lemma~\ref{le:main2} and Equations~\eqref{eq:noni1} and~\eqref{eq:noni2} it follows that
	\[
	2 n^{1/12}\frac{\norm{\bm{n}}_2}{\norm{\bm{n}}_1} + \frac{2}{3V_n} = O_p(n^{-1/12}).
	\]
	This immediately yields the claimed limits.
\end{proof}

Compare also with~\cite[Cor. 6.1]{MR3729639}.  We prove the following Prokhorov-approximation argument  by analogous arguments as for \cite[Cor. 6.2]{MR3729639}.

\begin{comment}
\begin{proposition}
	Consider (by abuse of notation) the uniform measure $\mu_{\cM(\mC_n)}$ on the $3$-connected core $\cM(\mC_n)$ as a Borel measure on the vertex set of the rescaled graph
	\[
	\left(\mC_n, \frac{3^{1/4}}{2 c^\dagger c_{\mathrm{fpp}} } (\kappa n)^{-1/4} d_{\mC_n}\right).
	\]
	Then
	\[
		d_{\mathrm{P}}(\mu_{\mC_n}, \mu_{\cM(\mC_n)}) \convp 0
	\]
	as $n \in 2\ndN$ tends to infinity.
\end{proposition}
\end{comment}
\begin{proposition}
	\label{pro:louyut}
	Set $\bm{n} = (n_i)_{1 \le i \le 3 V_n / 2}$ with $n_i = |\cD_i(\mC_n)|$  for all $1 \le i \le 3 V_n / 2$.	Then
	\[
	d_{\mathrm{P}}(\nu_{\cM(\mC_n)}^{\bm{n}}, \mu_{\cM(\mC_n)}) \convp 0
	\]
	as measures on the rescaled graph
	\[
		\left(\cM(\mC_n), \frac{3^{1/4}}{2 c^\dagger c_{\mathrm{fpp}} } (\kappa n)^{-1/4} d_{\mC_n}\right).
	\]
\end{proposition}
\begin{proof}
	  Let $(U_i)_{i \ge 1}$ be independent points of $\mathbf{M}$ with law $\mu_{\mathbf{M}}$. We stress that these points are generated in a quenched way. That is, we first generate $(\mathbf{M}, d_{\mathbf{M}}, \mu_{\mathbf{M}})$ and then sample independent points with respect to the same instance of the randomly generated measure $\mu_{\mathbf{M}}$.

	   Let $\epsilon>0$ be given. Since $\mathbf{M}$ is compact, it has a finite cover by open $(\epsilon/2)$-balls $C_1, \ldots, C_\ell$, $\ell \ge 1$. Since $\mu_{\mathbf{M}}$ has full support, it holds that
	  $\min_{1 \le i \le k} \mu_{\bm{M}}(C_i) > 0$. Hence, conditional on $(\mathbf{M}, d_{\mathbf{M}}, \mu_{\mathbf{M}})$, with probability one there exists an index $k$ such that for  all $1 \le i \le \ell$
	  \[
	  	\{U_1, \ldots, U_{k}\} \cap C_i \ne \emptyset.
	  \]
	  This entails that  $B_{\epsilon}^{\mathbf{M}}(U_1), \ldots, B_{\epsilon}^{\mathbf{M}}(U_{k})$ is  a finite open cover of $\mathbf{M}$. Hence, almost surely
	  \[
	  	K_\infty := \inf \left\{ k \ge 1 \,\, \Big\vert\,\, \bigcup_{i=1}^k B_\epsilon^{\mathbf{M}}(U_i) = \mathbf{M} \right\} < \infty.
	  \]
	  Let $(u_i)_{i \ge 1}$ denote independent uniform random vertices of $\cM(\mC_n)$. Since $\cM(\mC_n)$ has finitely many vertices, almost surely
	  \[
	  	K_n := \inf \left\{ k \ge 1 \,\, \Big\vert\,\, \bigcup_{i=1}^k B_\epsilon(u_i) = \cM(\mC_n) \right\} < \infty.
	  \]
	  By Corollary~\ref{co:comeonalmost}, it follows that there exists a constant $K \in \ndN$ with 
	  \begin{align}
	  	\label{eq:knkepsilon}
	  	\Pr{K_n > K} < \epsilon
	  \end{align}
	  for all large enough $n \in 2\ndN$. We set
	  \[
	  	A_1 = B_\epsilon(u_1)
	  \]
	  and for all $k \ge 2$
	  \[
	   A_k = B_\epsilon(u_k) \setminus \bigcup_{i=1}^{k-1} B_\epsilon(u_i),
	  \]
	  so that $A_1, \ldots, A_{K_n}$ is a cover of $\cM(\mC_n)$ by disjoint  subsets.
	  
	  Suppose that
	  \[
	  	d_{\mathrm{P}}(\nu_{\cM(\mC_n)}^{\bm{n}}, \mu_{\cM(\mC_n)}) > \epsilon.
	  \]
	  Then there exists a subset $S \subset \cM(\mC_n)$ such that
	  \[
	  	S^{\epsilon} := \{v \in \cM(\mC_n) \mid d_{\mC_n}(v, S) < \epsilon\}
	  \]
	  satisfies
	  \[
	  		\mu_{\cM(\mC_n)}(S^{\epsilon}) < \nu_{\cM(\mC_n)}^{\bm{n}}(S) - \epsilon.
	  \]
	  Since $A_1, \ldots, A_{K_n}$ is a disjoint cover of $\cM(\mC_n)$, it follows that
	  \[
	  	\sum_{j=1}^{K_n} \mu_{\cM(\mC_n)}(S^{\epsilon} \cap A_j) < \sum_{j=1}^{K_n} (\nu_{\cM(\mC_n)}^{\bm{n}}(S) - \epsilon/K_n).
	  \]
	  Hence there exists $1 \le j \le K_n$ with
	  \begin{align*}
	  		\mu_{\cM(\mC_n)}(S^{\epsilon} \cap A_j) < \nu_{\cM(\mC_n)}^{\bm{n}}(S \cap A_j) - \epsilon / K_n
	  \end{align*}
	  This implies $A_j \cap S \ne \emptyset$. Since $A_j$ is contained in an open $\epsilon$-ball, it follows that 
	  \[
	  	A_j \subset S^\epsilon.
	  \]
	  It follows that
	  \[
	  \mu_{\cM(\mC_n)}(A_j) < \nu_{\cM(\mC_n)}^{\bm{n}}(A_j) - \epsilon / K_n.
	  \]
	  Consequently, using~\eqref{eq:knkepsilon} it follows that
	  \begin{align*}
	  	  &\Prb{ d_{\mathrm{P}}(\nu_{\cM(\mC_n)}^{\bm{n}}, \mu_{\cM(\mC_n)}) > \epsilon  } \\
	  	  &\le \Prb{  |\mu_{\cM(\mC_n)}(A_j) - \nu_{\cM(\mC_n)}^{\bm{n}}(A_j)| > \epsilon / K_n \text{ for some } 1 \le j \le K_n } \\
	  	  &\le \sum_{j=1}^{K} \Prb{ |\mu_{\cM(\mC_n)}(A_j) - \nu_{\cM(\mC_n)}^{\bm{n}}(A_j)| > \epsilon / K_n, K_n \le K} + \Pr{K_n >K} \\
	  	  &\le \sum_{j=1}^{K} \Prb{ |\mu_{\cM(\mC_n)}(A_j) - \nu_{\cM(\mC_n)}^{\bm{n}}(A_j)| >  \epsilon / K} + \epsilon.
	  \end{align*}
	Moreover, for all $1 \le k \le K$
	\begin{align*}
		&|\mu_{\cM(\mC_n)}(A_k) - \nu_{\cM(\mC_n)}^{\bm{n}}(A_k)| \\
		&= \left| \mu_{\cM(\mC_n)}\left(B_\epsilon(u_k) \setminus \bigcup_{i=1}^{k-1} B_\epsilon(u_i)\right) - \nu_{\cM(\mC_n)}^{\bm{n}}\left(B_\epsilon(u_k) \setminus \bigcup_{i=1}^{k-1} B_\epsilon(u_i)\right)\right| \\
		&\le |\mu_{\cM(\mC_n)}(B_\epsilon(u_k)) - \nu_{\cM(\mC_n)}^{\bm{n}}(B_\epsilon(u_k))| \\
		 &\quad\,+ \sum_{i=1}^{K}|\mu_{\cM(\mC_n)}(B_\epsilon(u_k)\cap B_\epsilon(u_i)) - \nu_{\cM(\mC_n)}^{\bm{n}}(B_\epsilon(u_k) \cap B_\epsilon(u_i))|.
	\end{align*}
	By Corollary~\ref{co:noffff}, it follows that
	\[
		|\mu_{\cM(\mC_n)}(A_k) - \nu_{\cM(\mC_n)}^{\bm{n}}(A_k)| \convp 0.
	\]
	It follows that
	\[
		\limsup_{\substack{ n \to \infty \\ n \in 2\ndN}} \Prb{ d_{\mathrm{P}}(\nu_{\cM(\mC_n)}^{\bm{n}}, \mu_{\cM(\mC_n)}) > \epsilon  }  \le \epsilon.
	\]
	This holds for arbitrary $\epsilon>0$, hence the proof is complete.
\end{proof}

We are now ready to prove our main theorem.

\begin{proof}[Proof of Theorem~\ref{te:main}]
		Set $\bm{n} = (n_i)_{1 \le i \le 3 V_n / 2}$ with $n_i = |\cD_i(\mC_n)|$  for all $1 \le i \le 3 V_n / 2$. 
		
		Let $u$ be a uniformly at random selected vertex of $\mC_n$. If $u$ belongs to the $3$-connected core $\cM(\mC_n)$, then we set $v= u$. If $u$ does not, then it belongs to a unique network component, and we set $v$ to a uniformly selected endpoint of the edge of $\cM(\mC_n)$ where this component is inserted. Hence,  the law of $u$ is $\mu_{\mC_n}$, and the law $\mu_{\mathrm{mix}}$ of $v$ is a mixture of  $\mu_{\cM(\mC_n)}$ and~$\nu_{\cM(\mC_n)}^{\bm{n}}$.
		
		By Corollary~\ref{co:comeonalmost} and Proposition~\ref{pro:louyut}
it follows that
	\begin{align}
	\left(\cM(\mC_n), \frac{3^{1/4}}{2 c^\dagger c_{\mathrm{fpp}} } (\kappa n)^{-1/4} d_{\mC_n}, \mu_{\mathrm{mix}}\right) \convdis (\mathbf{M}, d_{\mathbf{M}}, \mu_{\mathbf{M}})
	\end{align}
	in the Gromov--Hausdorff--Prokhorov sense. 
	
	Let $R$ be the correspondence between $\left(\cM(\mC_n), \frac{3^{1/4}}{2 c^\dagger c_{\mathrm{fpp}} } (\kappa n)^{-1/4} d_{\mC_n}\right)$ and $\left(\mC_n, \frac{3^{1/4}}{2 c^\dagger c_{\mathrm{fpp}} } (\kappa n)^{-1/4} d_{\mC_n}\right)$ where for each edge of $\cM(\mC_n)$ its endpoints correspond  to all vertices of the corresponding network component (and to themselves) and vice versa. By 
	Lemma~\ref{le:compdiam} it follows that
	\begin{align}
		\mathrm{dis}(R) \convp 0.
	\end{align}
	The random points $u$ and $v$ correspond to each other with respect to $R$. By Proposition~\ref{pro:ghpcoupling}, it follows that
	\begin{align}
	\left(\mC_n, \gamma n^{-1/4} d_{\mC_n}, \mu_{\mC_n} \right) \convdis (\mathbf{M}, d_{\mathbf{M}}, \mu_{\mathbf{M}})
\end{align}
with
\begin{align}
	\gamma := \frac{3^{1/4}}{2 c^\dagger c_{\mathrm{fpp}} } \kappa ^{-1/4}.
\end{align}
\end{proof}

\begin{comment}
\appendix

\section{first-passage percolation on simple triangulations}

This appendix is dedicated to provide a proof of Lemma~\ref{le:left} that treats first-passage percolation on the dual of random simple triangulation. That is, given a random a	random positive integer $\iota$ with finite exponential moments, there exists a constant $c_E>0$ such that
	\begin{align}
		\label{eq:thefppappendix}
		n^{-1/4} \sup_{\substack{u,v \in \mQ_n\\ u\triangleleft f, v\triangleleft g }} \left | c_E d_{\mQ_n}(u,v) -  d_{\mathrm{fpp}}^\dagger(f,g)\right | \convp 0.
	\end{align}
A similar statement was proven in~\cite[Thm. 3]{zbMATH07144469} for uniform unrestricted triangulations with $\iota$-weights following an exponential distribution and for constant $\iota$-weights.

The proof of Lemma~\ref{le:left} we present here is by extending the arguments of~\cite[Thm. 3]{zbMATH07144469} to simple triangulations.
\end{comment}

\label{sec:app}

\bibliographystyle{abbrv}
\bibliography{bmcubic}

\end{document}